\documentclass[11pt]{article}
\usepackage[T1]{fontenc}
\usepackage[utf8]{inputenc}
\usepackage{lmodern}
\usepackage{amsmath,amssymb,amsthm}
\usepackage{geometry}
\usepackage{booktabs}
\usepackage{enumitem}
\usepackage[hidelinks]{hyperref}

\newtheorem{theorem}{Theorem}[section]
\newtheorem{lemma}[theorem]{Lemma}
\newtheorem{proposition}[theorem]{Proposition}
\newtheorem{corollary}[theorem]{Corollary}
\newtheorem{maintheorem}{Theorem}

\newtheorem{maincorollary}[maintheorem]{Corollary}
\theoremstyle{definition}
\newtheorem{definition}[theorem]{Definition}
\newtheorem{example}[theorem]{Example}
\theoremstyle{remark}

\newcommand{\F}{\mathbb{F}}
\newcommand{\Z}{\mathbb{Z}}
\newcommand{\GR}{\mathrm{GR}}
\newcommand{\Tr}{\mathrm{Tr}}
\newcommand{\ol}[1]{\overline{#1}}
\newcommand{\Dexc}{\Delta_{\mathrm{exc}}}
\newcommand{\Dunit}{\Delta_{\mathrm{unit}}}

\title{Coefficient choice in Galois-ring lifts of APN permutations:\\ sharp unit-direction spectra and one-term optimal lifts}
\author{Hiroki Fukui\thanks{Research Institute of Criminal Psychiatry / Sex Offender Medical Center, Tokyo, Japan;
and Department of Neuropsychiatry, Kyoto University, Kyoto, Japan. \texttt{fukui@somec.org}.
ORCID: 0009-0008-7122-522X.}}
\date{27 October 2026}

\begin{document}
\maketitle

\begin{abstract}
Permutation polynomials over the Galois ring $\GR(2^k,m)$ are nonlinear components for symmetric primitives
that add keys in the ring, as in secure computation modulo $2^k$. Rønjom and Sandrib showed that
their differential uniformity is at least $2q^{k-1}$, where $q=2^m$. For the class studied here it is attained on top-layer differences, while almost all differences are units. We ask how small the difference table can be on unit differences
when the residue permutation is fixed and only higher coefficient digits of its lift are chosen. Let $m\ge5$ be
odd, $k\ge3$, and let $p$ be a reduced quadratic almost perfect nonlinear (APN) permutation of $\F_q$ whose derivative
is nowhere zero and two-to-one. Over all lifts of $p$, the least possible maximum on unit differences is $q$ or $2q$,
decided by two invariants of $p$ and the level $k$. We describe every lift attaining it and show that the coefficients
modulo $8$ decide optimality at every level. For every odd $m\ge5$, residue permutations with at most $16$ monomials
attain $q$ after one $2$-adic digit of one coefficient of the natural lift is changed. For an explicit infinite family,
one added monomial lowers the maximum on unit differences from $4q$, or from at least $q^2$, to $2q$ at every level
without raising the degree. Under a uniform starting state and independent uniform round keys, valuation preservation yields layerwise bounds
for characteristics and two-round differentials that distinguish lifts with the same global uniformity. Core results are formalized in Lean~4.
\end{abstract}



\section{Introduction}\label{sec:intro}

Let $R_k=\GR(2^k,m)$ be the Galois ring of characteristic $2^k$ whose residue field is $\F_q$, $q=2^m$. Arithmetic in
$\Z/2^k$ and in Galois rings over it is the native arithmetic of several protocols for secure multiparty
computation~\cite{SPDZ2k,ACDEY}. In such settings additions are cheap and nonlinear operations are ring
multiplications, and over prime fields the same cost model has led to symmetric primitives built from low-degree
polynomials~\cite{MiMC}. A primitive over $R_k$ in this spirit adds round keys in $R_k$ and uses a permutation
polynomial of $R_k$ as its nonlinear layer. Difference propagation through this nonlinear layer is described by its difference table with respect to addition in
$R_k$. The largest nontrivial entry is its differential uniformity in the sense of Nyberg~\cite{Nyberg}, a basic
parameter in differential cryptanalysis~\cite{BiSh}.

Rønjom and Sandrib~\cite{RS} proved that every permutation polynomial of $R_k$ has differential uniformity at least
$2q^{k-1}$, and concluded that polynomial nonlinear layers over Galois rings are far from optimal. For the permutations considered below, their lower bound is
attained on the top layer $2^{k-1}R_k$ of input differences, which contains a fraction $(q-1)/(q^k-1)$ of the nonzero
differences. Almost all differences are units, and every lift considered here keeps the $2$-adic valuation of a
difference (Proposition~\ref{prop:valuation}). The behaviour on unit differences is therefore a separate design quantity, and so is
the behaviour on each intermediate layer. For lifts whose residue is APN, the general upper bound on unit differences is
$2q^{k-1}$~\cite[Prop.~4(1)]{RS}.

This paper determines how far the choice of coefficients can improve the unit differences. We fix a reduced quadratic
APN permutation $p$ of $\F_q$ satisfying hypothesis (H) below, where reduced means degree less than $q$, and consider its lifts: the
polynomials over $R_k$ that reduce to $p$ coefficient by coefficient. We determine the least possible largest entry of
the difference table on unit differences, all lifts attaining it, and the cost of attaining it. All lifts induce the same permutation of $\F_q$ and have the same differential uniformity $2q^{k-1}$
(for $\alpha\ne0$ the solutions of $F(x+\alpha)-F(x)=\beta$ lie in at most two residue classes~\cite{RS,Fukui1}). On
unit differences they behave very differently: depending on the lift, the largest entry ranges from the minimum ($q$ or $2q$) to
$q^2$ and more.

\paragraph{Terms.}
A \emph{lift} of $p$ is a polynomial $F\in R_k[X]$ whose coefficients reduce to those of $p$; the \emph{natural lift}
replaces each coefficient by the element of $R_k$ with the same binary coordinates in the polynomial basis
$1,t,\dots,t^{m-1}$ (Section~\ref{sec:prelim}). We call $k$ the \emph{level}. We write $\delta_F(\alpha,\beta)$ for the number of
$x\in R_k$ with $F(x+\alpha)-F(x)=\beta$. We assume that $p'$ is nowhere zero and two-to-one, with $p'(a)=p'(b)$ exactly
when $b\in\{a,a+\rho\}$ for a fixed $\rho\ne0$ (hypothesis (H), Section~\ref{sec:prelim}). A unit difference $\alpha$ is
\emph{exceptional} if its residue is $\rho$ and \emph{ordinary} otherwise. Every lift has entries at most $2$ in
ordinary directions (Lemma~\ref{lem:ordinary}), so everything happens in the exceptional directions, and we write
$\Dexc(F)$ for the largest entry there. Differences in $2^sR_k\setminus2^{s+1}R_k$ form the \emph{layer} $s$; the unit
differences are layer $0$ and the top layer is $s=k-1$.

\paragraph{Results.}
The minimum of $\Dexc$ over all lifts, of any degree, is $q$ or $2q$ (Theorem~\ref{thm:minimax}). Which value occurs
depends on whether $k=3$ or $k\ge4$ and on two invariants of $p$: the quadratic coefficient $\mu$ of the difference
$p(X+\rho)+p(X)$, and a zero set $Z\subseteq\F_q$ (both defined in Section~\ref{sec:main}). Theorem~\ref{thm:minimizers}
characterizes every lift attaining the minimum. Its criterion uses a function $A_F$ on $\F_q$, which decides whether
the solutions over a residue class split into several classes, and, when $\mu\ne0$, a bit that decides whether the
solutions coming from the two residues $a$ and $a+\rho$ overlap. Optimality can be read off from the coefficients modulo
$8$ (Theorem~\ref{thm:spectrum}, Corollary~\ref{cor:crosslevel}): a minimizer modulo $8$ remains optimal at every higher
level exactly when $A_F$ has no zero, and a family of lifts that is optimal at every level exists exactly when $\mu=0$ or
$Z\ne\emptyset$. The minimum is also cheap to reach. For the infinite family $P_{m,r}$ of quadratic APN permutations, one
added monomial gives an optimal lift at every level without raising the degree, and for $r\ge5$ a change of one
coefficient digit does the same (Theorem~\ref{thm:oneterm}, Proposition~\ref{prop:lowcost}). For every odd $m\ge5$ there
are residue permutations with at most $16$ monomials for which the minimum is $q$, and a suitable change of one $2$-adic
digit of one coefficient of the natural lift attains it (Theorems~\ref{thm:minimax} and~\ref{thm:sparse},
Proposition~\ref{prop:lowcost}).

For example, let $r=3$ and $3\nmid m$. The natural lift $F$ of $P_{m,3}$ has $\Dexc(F)=4q$. The lift
$F+2X^{4+2^{m-1}}$ has $\Dexc=2q$, and no lift of $P_{m,3}$ does better. A unit input difference then propagates with
probability at most $2/q^{k-1}$, which is $q^{k-2}$ times smaller than the general unit bound.

The proofs are local. Along a residue class the difference equation reduces to the quadratic map
$y\mapsto\mu y^2+A_F(a)y$ over $\F_q$, and one second-order invariant records how the solutions over $a$ and $a+\rho$
are shifted against each other. The same analysis gives the full spectrum of every exceptional row when $A_F$ has no
zero, an exact formula for the penultimate layer, and the intermediate layers. It also shows a conflict: when $\mu\ne0$,
optimizing the middle layer at $k=3$ forces the exceptional directions at every $k\ge4$ far from the optimum.

\paragraph{Cryptographic consequences.}
The choice of coefficients becomes a design parameter with a sharp target and a certificate that can be checked
modulo $8$. Since lifts preserve valuations, the bound for each layer combines with the branch number $B$ of the reduced
linear layer as in the wide-trail argument~\cite{DRwide,DR}, under a uniform starting state and independent uniform round
keys added in $R_k$ (Corollary~\ref{cor:crypto}, Proposition~\ref{prop:layerchar}). The top layer sets the bound over
all inputs: for every lift, every differential over two or more rounds has probability at most $(2/q)^{B-1}$, and the
bound for characteristics with input in the top layer is $(2/q)^{B\lfloor N/2\rfloor}$. The coefficient choice improves the
bounds for inputs below the top layer. In an eight-round example with four $21$-bit words and a linear layer of branch
number $5$, the bound for characteristics with a unit input word improves from $2^{-140}$ for the natural lift to
$2^{-260}$ for the lift $F^\sharp$ of $P_{7,5}$ with one added monomial, while the top layer gives $2^{-120}$ for every
lift. The same valuation argument bounds two-round differentials, without the factor up to $2^n$ that a count over the
active words would introduce for $n$ words: a differential whose input has a unit word has probability at most
$2^{-52}$ for $F^\sharp$ against $2^{-28}$ for the natural lift, and the bound over all differentials is $2^{-24}$ for
both (Propositions~\ref{prop:hull} and~\ref{prop:hullunit}). Markov's inequality then bounds the proportion of keys for
which a fixed differential is large (Proposition~\ref{prop:fixedkey}). A fixed plaintext breaks the product formula
(Example~\ref{ex:fixed}), so the probability model matters.

These results concern the polynomial part of a design. A cipher built only from ring polynomials, additive linear layers
and keys added in $R_k$ maps inputs that are congruent modulo $2$ to outputs that are congruent modulo $2$
(Proposition~\ref{prop:congpart}), so a complete design contains a further component that breaks this partition. Our theorems certify the polynomial segments and the coefficients
chosen inside them.

\paragraph{Relation to other work.}
This paper is a substantially revised and extended version of the author's preprint~\cite{Fukui2}, from which the
family $P_{m,r}$ and its analysis on ordinary directions and on the top layer are taken. The minimax theorem, the
classification of minimizers, the spectra and their certification modulo $8$, the results across levels and layers, the
optimal lifts of sparse representatives and the differential bounds are new. The earlier paper~\cite{Fukui1} constructs
optimal reduced representatives of quadratic APN permutations in odd dimension; here the representative is fixed and the
coefficients of its lift are optimized. A separate manuscript of the author, currently under review, leaves the
polynomial setting and studies digit-exchanged permutations of $\GR(4,m)$; it shares the lifting framework and the
basic directional estimates with the present paper. Bartoli and Stănică~\cite{BS} study when reduced polynomial lifts of
APN permutations are permutations. They give infinite liftable families, including a trace family of degree $q/2+2$
with a two-to-one derivative~\cite[Thm.~2.4]{BS}, and a liftability criterion for reduced Gold-power
representatives~\cite[Thm.~6.1, Cor.~6.2]{BS}. The residue polynomials $P_{m,r}$ fall under their Corollary~6.2, and
$P_{m,3}$ is their Remark~6.3. They do not consider the difference tables of individual lifts. Their trace family
satisfies (H), and we apply our theorems to it in Section~\ref{sec:discussion}. The dependence of the difference table on the
lift is easy to see: over $\Z/8$ the permutations $X$ and $X+2X^2$ have the same residue but different differences in
direction $1$. What is new here is the sharp minimum for a fixed residue, its minimizers, and their certification and
construction.

\paragraph{Verification.}
The central results are formalized in Lean~4, and the finite computations behind the examples and the existence proof
for the sparse family can be replayed from Online Resource~1. Appendix~\ref{app:verification}
lists what is formalized and what remains on paper.

\paragraph{Organization.}
Section~\ref{sec:main} states the main results. Section~\ref{sec:prelim} sets up notation and the local expansion along a
residue class. Section~\ref{sec:invariants} introduces the lift invariants and shows how to prescribe them.
Section~\ref{sec:spectra} proves the exceptional spectra and the lower bounds, and Section~\ref{sec:minimax} the minimax
theorem, the classification of minimizers and optimality across levels. Section~\ref{sec:family} treats $P_{m,r}$,
Section~\ref{sec:layers} the lower layers, and Section~\ref{sec:crypto} the cryptographic consequences.
Section~\ref{sec:discussion} discusses residue representatives, the trace family of~\cite{BS} and open problems.
Appendix~\ref{app:verification} describes the computations, the formalization and the use of AI tools.

\section{Main results}\label{sec:main}

Throughout, $m\ge5$ is odd, $q=2^m$, $k\ge3$ and $R_k=\GR(2^k,m)$. We fix a reduced quadratic APN
permutation $p\in\F_q[X]$ satisfying (H) of Section~\ref{sec:prelim}. In particular $p'=\varepsilon+K$ with $K$
linearized, $\ker K=\{0,\rho\}$, $\rho\ne0$ and $\varepsilon\notin K(\F_q)$. A \emph{lift} of $p$ is any
$F\in R_k[X]$ whose coefficients reduce to those of $p$. We impose no bound on the degree of $F$. A direction
$\alpha\in R_k$ is \emph{exceptional} if $\ol\alpha=\rho$, and
\begin{gather*}
  \Dexc(F)=\max\{\delta_F(\alpha,\beta):\ \ol\alpha=\rho,\ \beta\in R_k\},\\
  \delta_F(\alpha,\beta)=\#\{x\in R_k: F(x+\alpha)-F(x)=\beta\}.
\end{gather*}
Two invariants depend only on $p$:
\[
  \mu=[X^2]\bigl(p(X+\rho)+p(X)\bigr),\qquad
  Z=\{a\in\F_q:\ p(a+\rho)+p(a)+\rho\,p'(a)=0\}.
\]
Two invariants depend on the lift. For $\ol x=a$ and $\ol\tau=\rho$ put
\begin{gather*}
  A_F(a)=\ol{\tfrac12\bigl(F'(x+\tau)-F'(x)\bigr)},\\
  E_F(a)=\ol{\tfrac14\bigl(F(x+2\tau)-2F(x+\tau)+F(x)\bigr)}\quad(a\in Z).
\end{gather*}
Both are well defined (Lemmas~\ref{lem:A} and~\ref{lem:E}). When $\mu\ne0$, $a\in Z$ and $A_F(a)\ne0$ we set
$\chi_F(a)=\Tr\bigl(\mu E_F(a)/A_F(a)^2\bigr)$.

\begin{maintheorem}[A one-term optimal lift]\label{thm:oneterm}
Let $m\ge5$ be odd, $3\le r<m$ odd with $\gcd(r,m)=1$, and let $F$ be the natural $0/1$ lift of the explicit APN
permutation $P_{m,r}$ of Section~\ref{sec:family}. Put
\[
  F^\sharp=F+2X^3\ (r\ge5),\qquad F^\sharp=F+2X^{4+2^{m-1}}\ (r=3).
\]
For every $k\ge3$ every exceptional row of $F^\sharp$ consists of exactly $q^{k-1}/2$ entries equal to $2q$
and zeros, and every ordinary unit row has entries at most $2$. Thus $\Dunit(F^\sharp)=2q$, which is the least
value of $\Dexc$ over all lifts of $P_{m,r}$. The correction does not raise the degree $3q/4$ and, in the termwise evaluation model, costs one ring
multiplication and two additions on top of the squaring schedule of $F$; for $r\ge5$ a change of one coefficient
digit gives the same spectra with no extra multiplication (Proposition~\ref{prop:lowcost}). The natural lift itself has
$\Dexc=4q$ for $r=3$ and $\Dexc\ge q^2$ for $r\ge5$ at every level; at $k=3$ and $r\ge5$ its value is exactly $q^2$ if
$r\equiv1\pmod4$ and $2q^2$ if $r\equiv3\pmod4$ (Corollary~\ref{cor:natural}).
\end{maintheorem}

\begin{maintheorem}[Minimax over all lifts]\label{thm:minimax}
The smallest possible value of $\Dexc(F)$ over all lifts $F$ of $p$ is
\begin{center}
\begin{tabular}{lcc}
\toprule
 & $k=3$ & $k\ge 4$\\ \midrule
 $\mu=0$, $Z=\emptyset$ & $q$ & $q$\\
 $\mu=0$, $Z\neq\emptyset$ & $2q$ & $2q$\\
 $\mu\neq0$, $Z=\emptyset$ & $q$ & $2q$\\
 $\mu\neq0$, $Z\neq\emptyset$ & $2q$ & $2q$\\ \bottomrule
\end{tabular}
\end{center}
\end{maintheorem}

\begin{maintheorem}[All minimizers]\label{thm:minimizers}
Let $k\ge4$. If $\mu=0$, a lift is optimal if and only if $A_F$ has no zero.
If $\mu\neq0$, a lift is optimal if and only if $A_F$ has no zero and $\chi_F\equiv1$ on $Z$.
For $k=3$ the optimal lifts are those with $A_F$ zero-free when $\mu=0$,
those with $A_F\equiv0$ when $\mu\neq0$ and $Z=\emptyset$,
and those with $\chi_F(a)=1$ at every $a\in Z$ with $A_F(a)\neq0$ when $\mu\neq0$ and $Z\neq\emptyset$.
\end{maintheorem}

\begin{maintheorem}[Exceptional spectra and certification modulo $8$]\label{thm:spectrum}
Let $n_c$ be the number of $\beta$ with $\delta_F(\alpha,\beta)=c$ for a fixed exceptional $\alpha$.
\begin{enumerate}[label=(\arabic*)]
\item If $A_F$ has no zero, then $\delta_F(\alpha,\beta)=\delta_{F\bmod 8}(\alpha\bmod 8,\beta\bmod 8)$.
 For $\mu=0$ we get $n_q=(q-|Z|)q^{k-2}$ and $n_{2q}=\tfrac12|Z|q^{k-2}$.
 For $\mu\neq0$ we get $n_{2q}=\tfrac12(q-\nu)q^{k-2}$ and $n_{4q}=\tfrac14\nu q^{k-2}$,
 where $\nu=\#\{a\in Z:\chi_F(a)=0\}$.
\item If $k=3$ and $\mu\neq0$, then for arbitrary $A_F$
 \[ n_q=(s-z_0)q,\quad n_{2q}=\tfrac12(q-s+z_0-\nu)q,\quad n_{4q}=\tfrac14\nu q, \]
 with $s=\#A_F^{-1}(0)$, $z_0=\#(Z\cap A_F^{-1}(0))$ and $\nu=\#\{a\in Z: A_F(a)\neq0,\ \chi_F(a)=0\}$.
\end{enumerate}
All other positive values have multiplicity zero.
\end{maintheorem}

\begin{maincorollary}[Optimality across levels]\label{cor:crosslevel}
The criteria in Theorem~\ref{thm:minimizers} determine optimality at every higher level from the
coefficient data modulo $8$. A minimizer at level $k=3$ (that is, modulo $8$) remains optimal at all higher levels if and only if its
function $A_F$ is nowhere zero. Consequently, a compatible family minimizing the exceptional uniformity at
every level $k\ge3$ exists if and only if $\mu=0$ or $Z$ is nonempty.
\end{maincorollary}

\begin{maintheorem}[Penultimate layer and compatibility]\label{thm:layers}
Let $s=\#A_F^{-1}(0)$. For $k\ge4$ and $\alpha\in 2^{k-2}R_k\setminus 2^{k-1}R_k$, every row has $q-s$ entries $q^{k-1}$ and $s/2$
entries $2q^{k-1}$. Optimality of the $k=3$ middle layer $2R_3\setminus4R_3$ (all nonzero entries equal to $q^2$) requires $A_F\equiv0$ when $\mu\neq0$,
which forces $\Dexc\ge q^2$ at every level $k\ge4$.
When $\mu=0$ a zero-free $A_F$ optimizes all three at once.
\end{maintheorem}

\begin{maincorollary}[Transition probabilities]\label{cor:crypto}
For $F^\sharp$ the maximal transition probability from a unit input difference is $2/q^{k-1}$, a factor
$q^{k-2}$ below the general bound of Rønjom and Sandrib. Differences keep their $2$-adic valuation.
Consider $N$ rounds $x\mapsto M\cdot S(x)+\kappa_i$ on $R_k^n$, where $S$ applies a lift $F$ of $p$ to each word and
$M$ is an additive automorphism of $R_k^n$ whose reduction $\ol M$ has differential branch number $B$. Take the
starting state and the round keys uniform and all independent, with the keys added in $R_k$.
\begin{enumerate}[label=(\alph*)]
\item If $F=F^\sharp$, a fixed characteristic whose input difference has a unit component has probability at most
$(2/q^{k-1})^{B\lfloor N/2\rfloor}$.
\item If $F=F^\sharp$ with $r\ge5$ and the input difference lies in $(2^sR_k)^n$ with a word outside $2^{s+1}R_k$,
$1\le s\le k-2$, the bound is $q^{(s+1-k)B\lfloor N/2\rfloor}$. For $s=k-1$ it is $(2/q)^{B\lfloor N/2\rfloor}$ for
every lift.
\item For every lift, every $N\ge2$ and every pair $(a,b)$ with $a\ne0$, the differential $a\to b$ has probability at
most $(2/q)^{B-1}$. If $F=F^\sharp$ and $a$ has a unit component, it has probability at most $(2/q^{k-1})^{B-1}$; if
$F=F^\sharp$ with $r\ge5$ and $a$ is in layer $s$ as in (b), $1\le s\le k-2$, at most $q^{(s+1-k)(B-1)}$.
\item For each such two-round differential with bound $\varepsilon$ from (c) and every $t>0$, at most a proportion
$\varepsilon/t$ of the key pairs gives it a fixed-key probability $\ge t$.
\end{enumerate}
\end{maincorollary}

\section{Preliminaries}\label{sec:prelim}

\subsection{Rings, lifts and direction classes}
Let $f\in\F_2[t]$ be irreducible of degree $m$ and $\tilde f$ its $0/1$ lift. Then
$R_k=(\Z/2^k)[t]/(\tilde f)$ and $\F_q=\F_2[t]/(f)$, with reduction $x\mapsto\ol x$. The \emph{natural lift} of
$p\in\F_q[X]$ replaces each coefficient by the $0/1$ lift of its coordinates in the basis $1,t,\dots,t^{m-1}$; it
depends on $f$. A nonzero
$\alpha\in R_k$ is \emph{ordinary} if $\ol\alpha\notin\{0,\rho\}$, \emph{exceptional} if $\ol\alpha=\rho$, and
lies in the layer $v_j$ if $\alpha\in2^jR_k\setminus2^{j+1}R_k$ ($1\le j\le k-1$). The layer $v_{k-1}$ is the
top layer and $v_{k-2}$ the penultimate layer. We write $\Delta_{\mathrm{cls}}(F)$ for the maximum of
$\delta_F(\alpha,\beta)$ over $\alpha$ in a class and all $\beta$, and $\Dunit=\max(\Delta_{\mathrm{ord}},\Dexc)$.

For $F\in R_k[X]$ we use Hasse derivatives $F^{[i]}$, defined by $F(x+h)=\sum_{i\ge0}F^{[i]}(x)h^i$ exactly in
$R_k$. Then $F^{[1]}=F'$ and $2F^{[2]}=F''$. For a function $g$ on $\F_q$ we write $D_\rho g(a)=g(a+\rho)+g(a)$.
Division by $2$ or $4$ is applied only after the numerator has been shown to lie in $2R_k$ or $4R_k$, and only
its residue is used.

\subsection{The hypothesis (H)}
\begin{definition}
A polynomial $p\in\F_q[X]$ satisfies (H) if it is reduced ($\deg p<q$), quadratic (only Dembowski--Ostrom
monomials $X^{2^i+2^j}$, linear monomials $X^{2^i}$ and a constant), an APN permutation of $\F_q$, and its
formal derivative is $p'=\varepsilon+K$ with $K$ linearized, $\ker K=\{0,\rho\}$, $\rho\ne0$ and
$\varepsilon\notin K(\F_q)$.
\end{definition}
For $m=3$ no polynomial satisfies (H): an exhaustive search over the reduced quadratic APN permutations of $\F_8$ finds
none with $\varepsilon\notin K(\F_8)$ and $\#\ker K=2$. This is why we assume $m\ge5$. Under (H), $p'$ has no zero and is
exactly two-to-one: $p'(a)=p'(b)$ if and only if $b\in\{a,a+\rho\}$. Every lift
of $p$ is a permutation of $R_k$, because its residue is a permutation and its derivative takes unit
values~\cite[Thm.~2]{RS},~\cite[Thm.~2.1]{BS}. Put $\Lambda(a)=p(a+\rho)+p(a)+\rho p'(a)$, so that
$Z=\Lambda^{-1}(0)$, and $z=|Z|$.

\begin{lemma}\label{lem:F0}
\begin{enumerate}[label=(\roman*)]
\item $p^{[2]}$ is affine, so $p^{[2]}(a+\rho)+p^{[2]}(a)=\mu$ for all $a$.
\item $\Lambda$ is $\rho$-periodic, so $Z$ is a union of pairs $\{a,a+\rho\}$.
\item $a\mapsto p(a+\rho)+p(a)$ is $\rho$-periodic and exactly two-to-one.
\end{enumerate}
\end{lemma}
\begin{proof}
(i) The second Hasse coefficient of $X^e$ is $\binom e2$. For $e=2^i+2^j$ with $i<j$ it is odd only for
$e=3$ or $i=1$, and for $e=2^i$ only for $e=2$. Hence $p^{[2]}=c_2+c_3X+\sum_{j\ge2}c_{2+2^j}X^{2^j}$ is affine.
(ii) $\Lambda(a+\rho)=p(a)+p(a+\rho)+\rho p'(a+\rho)=\Lambda(a)$ because $p'(a+\rho)=p'(a)$.
(iii) Periodicity is clear. Since $p$ is APN every value has at most two preimages, and periodicity gives
exactly the two preimages $a$ and $a+\rho$.
\end{proof}

\begin{lemma}[Ordinary directions]\label{lem:ordinary}
If $\alpha$ is ordinary, then $\delta_F(\alpha,\beta)\in\{0,2\}$.
\end{lemma}
\begin{proof}
The residue $\ol G=p(X+\ol\alpha)+p(X)+\ol\beta$ of $G=F(X+\alpha)-F(X)-\beta$ has at most two roots by the
APN property. At a root $\ol G'(a)=p'(a+\ol\alpha)+p'(a)=K(\ol\alpha)\ne0$. Simple roots lift uniquely by
Hensel's lemma.
\end{proof}

\subsection{The local expansion along a residue class}
Fix an exceptional $\alpha$ and $\beta\in R_k$ and put $G(X)=F(X+\alpha)-F(X)-\beta$.

\begin{lemma}[The branch gate]\label{lem:gate}
Let $\ol x=a$.
\begin{enumerate}[label=(\roman*)]
\item $G(x+2u)\equiv G(x)\pmod 4$ for all $u\in R_k$.
\item $G(x+\alpha)\equiv G(x)+2[\Lambda(a)]\pmod 4$, where $[\cdot]$ is any lift.
\end{enumerate}
Hence, if $G\in4R_k$ on one residue class of a pair, it lies in $4R_k$ on the other class if and only if
$\Lambda(a)=0$.
\end{lemma}
\begin{proof}
(i) $G(x+2u)=G(x)+2uG'(x)+4(\cdots)$ and $G'(x)=F'(x+\alpha)-F'(x)\in2R_k$, because its residue is
$p'(a+\rho)+p'(a)=0$.
(ii) $G(x+\alpha)-G(x)=S(x):=F(x+2\alpha)-2F(x+\alpha)+F(x)=\sum_{i\ge1}(2^i-2)\alpha^iF^{[i]}(x)$. The
term $i=1$ vanishes and $2^i\equiv0\pmod4$ for $i\ge2$, so
$S(x)\equiv-2\bigl(F(x+\alpha)-F(x)-\alpha F'(x)\bigr)\pmod4$. The bracket has residue $\Lambda(a)$.
\end{proof}

\begin{lemma}[Local formula]\label{lem:local}
Let $x_a$ be a fixed lift of $a$ with $G(x_a)\in4R_k$, put $c_a=\ol{G(x_a)/4}$ and
$L_a(y)=\mu y^2+A_F(a)y$.
\begin{enumerate}[label=(\roman*)]
\item Choose $g_0,g_1\in R_k$ with $G(x_a)=4g_0$ and $G'(x_a)=2g_1$ (possible by Lemma~\ref{lem:A}(ii)) and put
$\Phi_a(z)=g_0+g_1z+G^{[2]}(x_a)z^2+\sum_{i\ge3}2^{i-2}G^{[i]}(x_a)z^i$. Then $G(x_a+2z)=4\Phi_a(z)$, so
$G(x_a+2z)=0$ if and only if $\Phi_a(z)\in2^{k-2}R_k$. This condition depends only on $z$ modulo $2^{k-2}$, and
$\Phi_a$ reduces modulo $2^{k-2}$ to a polynomial over $R_k/2^{k-2}R_k\cong R_{k-2}$ whose residue is
$c_a+A_F(a)y+\mu y^2$ with $y=\ol z$. The class of $a$ is $\{x_a+2z\}$ with $z$ running modulo $2^{k-1}$.
\item For $k=3$ the number of solutions in the class of $a$ is $q\cdot\#\{y\in\F_q:L_a(y)=c_a\}$.
\item If $\Lambda(a)=0$ and $x_{a+\rho}=x_a+\alpha$, then $c_{a+\rho}=c_a+E_F(a)$.
\end{enumerate}
\end{lemma}
\begin{proof}
(i) The residue of $G'(x_a)/2$ is $A_F(a)$ by definition, and that of $G^{[2]}(x_a)$ is
$p^{[2]}(a+\rho)+p^{[2]}(a)=\mu$ by Lemma~\ref{lem:F0}(i).
(ii) For $k=3$ the value $G(x_a+2z)/4$ modulo $2$ depends on $\ol z$ only, and each $\ol z$ has $q$ lifts
$z$ modulo $4$.
(iii) $G(x_a+\alpha)-G(x_a)=S(x_a)\in4R_k$ when $\Lambda(a)=0$, and $\ol{S(x_a)/4}=E_F(a)$.
\end{proof}

\section{Lift invariants}\label{sec:invariants}

\begin{lemma}[$A_F$]\label{lem:A}
\begin{enumerate}[label=(\roman*)]
\item $F'(x+2u)\equiv F'(x)\pmod4$ for all $x,u$.
\item $F'(x+\tau)-F'(x)\in2R_k$ when $\ol\tau=\rho$, and $A_F(a)$ does not depend on the lifts $x$, $\tau$.
\item $A_F$ is $\rho$-periodic.
\end{enumerate}
\end{lemma}
\begin{proof}
(i) $F'(x+2u)=F'(x)+2uF''(x)+\sum_{i\ge2}(2u)^i(F')^{[i]}(x)$ and $F''=2F^{[2]}$.
(ii) The residue of the difference is $p'(a+\rho)+p'(a)=0$. Changing $x$ to $x+2u$ or $\tau$ to $\tau+2v$
does not change either term modulo $4$ by (i).
(iii) $A_F(a)+A_F(a+\rho)$ is the residue of $(F'(x+2\tau)-F'(x))/2$, using $x+\tau$ as the lift of
$a+\rho$. By (i) the numerator lies in $4R_k$.
\end{proof}

\begin{proposition}[Prescribing $A_F$]\label{prop:Acontrol}
\begin{enumerate}[label=(\roman*)]
\item For $H\in R_k[X]$, the lift $F+2H$ has the same residue and $A_{F+2H}=A_F+D_\rho h$, where $h$ is the
residue function of $H'$.
\item Every function $h:\F_q\to\F_q$ is the residue function of $H'$ for some $H$.
\item $\{D_\rho h\}$ is the set of all $\rho$-periodic functions.
\end{enumerate}
Hence $A_F$ can be changed to any $\rho$-periodic function without changing the residue of $F$.
\end{proposition}
\begin{proof}
(i) The change of the numerator is $2(H'(x+\tau)-H'(x))$. (ii) Let $s=h^{1/2}$ pointwise, interpolate
$s=\sum s_eX^e$ and put $H=\sum_e[s_e^2]X^{2e+1}$. Then $H'\equiv\sum_es_e^2X^{2e}=s^2\pmod2$. (iii) A
$\rho$-periodic $g$ is $D_\rho h$ for $h(a)=0$, $h(a+\rho)=g(a)$ on a set of pair representatives.
\end{proof}

\begin{lemma}[$E_F$]\label{lem:E}
Let $a\in Z$, $\ol x=a$, $\ol\tau=\rho$ and $S(x)=F(x+2\tau)-2F(x+\tau)+F(x)$.
\begin{enumerate}[label=(\roman*)]
\item $S(x)\in4R_k$.
\item $\ol{S(x)/4}$ does not depend on the lifts $x$ and $\tau$.
\end{enumerate}
\end{lemma}
\begin{proof}
(i) As in Lemma~\ref{lem:gate}(ii), $S(x)\equiv-2(F(x+\tau)-F(x)-\tau F'(x))\pmod 4$ and the bracket has residue
$\Lambda(a)=0$.
(ii) Changing $x$ to $x+2u$ adds $2uS'(x)+4u^2S^{[2]}(x)+\sum_{i\ge3}(2u)^iS^{[i]}(x)$. By
Lemma~\ref{lem:A}(i), $S'(x)\equiv2(F'(x)-F'(x+\tau))\equiv0\pmod4$, and $S^{[2]}(x)$ has residue
$p^{[2]}(a)+p^{[2]}(a)=0$. So the change lies in $8R_k$. Changing $\tau$ to $\tau+2v$ changes $S$ modulo $8$
by $4v(F'(x+2\tau)-F'(x+\tau))$, and $F'(x+2\tau)\equiv F'(x)\equiv F'(x+\tau)\pmod 2$.
\end{proof}

\begin{proposition}[Translation identity]\label{prop:Etransl}
For $a\in Z$, $E_F(a+\rho)+E_F(a)=\rho A_F(a)+\rho^2\mu=L_a(\rho)$.
\end{proposition}
\begin{proof}
Put $G(X)=F(X+\tau)-F(X)$, so $S(X)=G(X+\tau)-G(X)$. The sum telescopes exactly:
$S(x+\tau)+S(x)=G(x+2\tau)-G(x)$. Modulo $8$, $G(x+2\tau)-G(x)\equiv2\tau G'(x)+4\tau^2G^{[2]}(x)$. Here
$G'(x)\in2R_k$ with $\ol{G'(x)/2}=A_F(a)$, and $\ol{G^{[2]}(x)}=\mu$. Both $a$ and $a+\rho$ lie in $Z$, and
$x+\tau$ is a lift of $a+\rho$ by Lemma~\ref{lem:E}(ii).
\end{proof}

\begin{lemma}[Overlap of images]\label{lem:overlap}
Put $A=A_F(a)$. (i) If $\mu\ne0$ and $A\ne0$, then $L_a$ is $\F_2$-linear and two-to-one with kernel
$\{0,A/\mu\}$, and $E\in L_a(\F_q)$ if and only if $\Tr(\mu E/A^2)=0$. (ii) If $\mu\ne0$ and $A=0$, then $L_a$
is bijective. (iii) If $\mu=0$ and $A\ne0$, then $L_a$ is bijective; if $\mu=A=0$, then $L_a\equiv0$. Moreover $\chi_F(a+\rho)=\chi_F(a)$ for $a\in Z$.
\end{lemma}
\begin{proof}
With $y=(A/\mu)t$ we get $L_a(y)=(A^2/\mu)(t^2+t)$, and $\{t^2+t\}=\ker\Tr$ (both have $q/2$ elements and
$\Tr(t^2+t)=0$). For the periodicity, $A_F(a+\rho)=A_F(a)$ and
$\Tr(\mu(E_F(a+\rho)+E_F(a))/A^2)=\Tr(s+s^2)=0$ with $s=\mu\rho/A$ by Proposition~\ref{prop:Etransl}.
\end{proof}

\begin{proposition}[Prescribing $E_F$]\label{prop:Econtrol}
For $J\in R_k[X]$ with residue function $j$, the lift $F+2J^2$ has the same residue and the same $A_F$, and
$E_{F+2J^2}(a)=E_F(a)+(D_\rho j(a))^2$ for $a\in Z$. Hence, keeping $A_F$, the function $E_F$ can be shifted on
$Z$ by any $\rho$-periodic function.
\end{proposition}
\begin{proof}
Put $Q=J^2$. Since $Q'=2JJ'$, we have $Q(x+2\tau)\equiv Q(x)\pmod4$. The change of $E$ is the residue of
$(Q(x+2\tau)-2Q(x+\tau)+Q(x))/2\equiv Q(x)-Q(x+\tau)\pmod2$, which is $j(a)^2+j(a+\rho)^2$. The change of $A$ is
$D_\rho$ of the residue of $Q'$, which is $0$. The last sentence follows from
Proposition~\ref{prop:Acontrol}(iii) and bijectivity of squaring.
\end{proof}

\section{Exceptional spectra and lower bounds}\label{sec:spectra}

\begin{theorem}\label{thm:S12}
Let $k=3$ and fix an exceptional $\alpha$.
\begin{enumerate}[label=(\roman*)]
\item If $\mu\ne0$, then $n_q=(s-z_0)q$, $n_{2q}=(q-s+z_0-\nu)q/2$ and $n_{4q}=\nu q/4$.
\item If $\mu=0$ and $A_F$ has no zero, then $n_q=(q-z)q$ and $n_{2q}=(z/2)q$.
\end{enumerate}
\end{theorem}
\begin{proof}
If $\ol\beta$ is not a value of $D_\rho p$ then $\delta_F(\alpha,\beta)=0$. If $\ol\beta=D_\rho p(a)$ only the
pair $\{a,a+\rho\}$ contributes (Lemma~\ref{lem:F0}(iii)). By Lemma~\ref{lem:gate}(i) the class of $a$
passes the gate for exactly one value of $\beta\bmod4$. The remaining digit of $\beta$ runs through $q$
values and is in bijection with $c_a$. The class of $a$ contributes $q\cdot\#L_a^{-1}(c_a)$
(Lemma~\ref{lem:local}(ii)).

If $\Lambda(a)\ne0$ the two classes pass for different $\beta\bmod4$ (Lemma~\ref{lem:gate}(ii)). For each
class: if $\mu\ne0=A(a)$ we get $q$ values $\beta$ with entry $q$. If $\mu\ne0\ne A(a)$ we get $q/2$ values
with entry $2q$. If $\mu=0\ne A(a)$ we get $q$ values with entry $q$.

If $\Lambda(a)=0$ both classes pass for the same $\beta\bmod4$, and $c_{a+\rho}=c_a+E_F(a)$
(Lemma~\ref{lem:local}(iii)). The entry is $q(\#L_a^{-1}(c)+\#L_a^{-1}(c+E))$, since $L_{a+\rho}=L_a$ by Lemma~\ref{lem:A}(iii). If $\mu\ne0=A$ we get $q$ values
with entry $2q$. If $\mu\ne0\ne A$ and $\chi=1$ the images are disjoint and we get $q$ values with entry $2q$.
If $\chi=0$ the images coincide and we get $q/2$ values with entry $4q$. If $\mu=0\ne A$ we get $q$ values
with entry $2q$.

Counting residues outside and inside $Z$ gives the formulas. For (i): $s-z_0$ residues outside $Z$ with
$A=0$ give $n_q$. The value $2q$ comes from $q-z-s+z_0$ residues outside $Z$ with $A\ne0$ (each $q/2$),
$z_0/2$ pairs in $Z$ with $A=0$, and $(z-z_0-\nu)/2$ pairs in $Z$ with $\chi=1$ (each $q$). The value $4q$
comes from $\nu/2$ pairs with $\chi=0$ (each $q/2$). As a check, $q\,n_q+2q\,n_{2q}+4q\,n_{4q}=q^3$.
\end{proof}

\begin{theorem}[Stability modulo $8$]\label{thm:S3}
If $A_F$ has no zero, then $\delta_F(\alpha,\beta)=\delta_{F_3}(\alpha_3,\beta_3)$ for every exceptional
$\alpha$ and every $\beta$, where the index $3$ denotes reduction modulo $8$. Consequently, for every
$k\ge3$, $n_q=(q-z)q^{k-2}$ and $n_{2q}=(z/2)q^{k-2}$ if $\mu=0$, and $n_{2q}=(q-\nu)q^{k-2}/2$ and
$n_{4q}=\nu q^{k-2}/4$ if $\mu\ne0$.
\end{theorem}
\begin{proof}
By Lemma~\ref{lem:local}(i), the solutions in the class of $a$ correspond to the $z$ modulo $2^{k-1}$ with
$\Phi_a(z)\equiv0\pmod{2^{k-2}}$, and the residue $c_a+A y+\mu y^2$ of $\Phi_a$ has constant derivative
$A=A_F(a)\ne0$. Every residue root is simple and lifts by Hensel's lemma over $R_{k-2}$ to exactly one root
modulo $2^{k-2}$, which has $q$ lifts modulo $2^{k-1}$. So each residue root gives $q$ solutions. The number of residue
roots depends only on $c_a$ (and on $c_a+E_F(a)$ when $\Lambda(a)=0$), hence only on $\beta\bmod8$. The same
count holds at $k=3$. Each $\beta\bmod 8$ has $q^{k-3}$ lifts. The invariants $A_F$ and $E_F$ depend only on $F$
modulo $8$.
\end{proof}

\begin{theorem}[Lower bounds]\label{thm:M1}
For every lift and every $k\ge3$:
(i) $\Dexc\ge q$; (ii) $\Dexc\ge2q$ if $Z\ne\emptyset$; (iii) $\Dexc\ge q^2$ if $k\ge4$ and $A_F$ has a zero;
(iv) $\Dexc\ge2q$ if $k\ge4$ and $\mu\ne0$.
\end{theorem}
\begin{proof}
(i) By Lemma~\ref{lem:gate}(i) the $q^{k-1}$ differences on a residue class lie in one coset of $4R_k$, which
has $q^{k-2}$ elements. (ii) If $\Lambda(a)=0$ the two classes of the pair, $2q^{k-1}$ points, land in the same
coset of $4R_k$. (iii) Let $A_F(a)=0$ and $\ol x=a$. Then $G_\alpha'(x)\equiv0\pmod4$. For $h\in2^{k-2}R_k$,
$G_\alpha(x+h)-G_\alpha(x)=hG_\alpha'(x)+\sum_{i\ge2}h^iG_\alpha^{[i]}(x)$. The first term lies in
$2^{k-2}\cdot4R_k=0$, and $h^i\in2^{2k-4}R_k=0$ for $k\ge4$. So the difference is constant on the $q^2$ points
of $x+2^{k-2}R_k$. (iv) By (iii) we may assume $A_F$ has no zero. Choose $\beta$ with $c_a=0$. Then
$L_a(y)=0$ has two roots, and by Theorem~\ref{thm:S3} each gives $q$ solutions.
\end{proof}

\section{Minimax, all minimizers and optimality across levels}\label{sec:minimax}

\begin{proof}[Proof of Theorem~\ref{thm:minimax}]
The lower bounds are Theorem~\ref{thm:M1}; for $k=3$ and $\mu\ne0$ only (i) and (ii) are used. For attainment:
if $\mu=0$, make $A_F\equiv1$ (Proposition~\ref{prop:Acontrol}); Theorems~\ref{thm:S12} and~\ref{thm:S3} give
the value $q(1+[Z\ne\emptyset])$. If $\mu\ne0$ and $k=3$, make $A_F\equiv0$; Theorem~\ref{thm:S12}(i) with
$s=q$, $z_0=z$, $\nu=0$ gives the value $q(1+[Z\ne\emptyset])$. If $\mu\ne0$ and $k\ge4$, make $A_F\equiv1$ and
then shift $E_F$ on $Z$ by $d=A_F^2(1+\chi_F)/\mu$ (Proposition~\ref{prop:Econtrol}). This $d$ is
$\rho$-periodic by Lemma~\ref{lem:overlap}, and since $\Tr(1)=1$ for odd $m$ it makes $\chi_F\equiv1$. Then
Theorem~\ref{thm:S3} gives $n_{2q}=q^{k-1}/2$ and the value $2q$.
\end{proof}

\begin{proof}[Proof of Theorem~\ref{thm:minimizers}]
Let $k\ge4$. If $A_F$ has a zero then $\Dexc\ge q^2>4q$ by Theorem~\ref{thm:M1}(iii). If $A_F$ has no zero,
Theorem~\ref{thm:S3} gives the full spectrum: for $\mu=0$ the maximum is the minimax value, and for $\mu\ne0$
it is $2q$ if $\nu=0$ and $4q$ otherwise.
Let $k=3$ and $\mu=0$. If $A_F(a)=0$, a $\beta$ with $c_a=0$ gives $q\cdot q=q^2$ solutions in the class of
$a$, more than $2q$. If $A_F$ has no zero, Theorem~\ref{thm:S12}(ii) gives the minimax value.
Let $k=3$ and $\mu\ne0$. By Theorem~\ref{thm:S12}(i) the maximum is $4q$ if $\nu>0$, otherwise $2q$ if
$n_{2q}>0$, otherwise $q$. If $Z=\emptyset$, then $n_{2q}=(q-s)q/2$ vanishes if and only if $A_F\equiv0$. If
$Z\ne\emptyset$, the value $2q$ is attained if and only if $\nu=0$.
\end{proof}

\begin{example}
A minimizer at $k=3$ need not satisfy the middle-layer condition of Theorem~\ref{thm:O3} in Section~\ref{sec:layers}.
For $m=5$ (field elements written as integers, bit $i$ = coefficient of $t^i$, $t^5=t^2+1$) let
$F_1=c_0+2X^2$, where $c_0$ is a lift with $A\equiv1$ of a representative with $\rho=21$, $\mu=28$ and
$Z=\{12,25\}$ (listed in Online Resource~1). Then $A_{F_1}\equiv1$ and $\chi_{F_1}\equiv1$, and all $1024$ exceptional
rows of $F_1$ consist of $512$ entries equal to $2q=64$.
\end{example}

\paragraph{Optimality across levels.}
Let $k\ge4$, let $F_k\in R_k[X]$ be a lift and let $F_3$ be its coefficientwise reduction modulo $8$, a
\emph{compatible} lift at level $3$. Every $F_3$ has compatible lifts at every level.

\begin{lemma}\label{lem:CL0}
$A_{F_k}=A_{F_3}$, $E_{F_k}=E_{F_3}$ on $Z$, $\chi_{F_k}=\chi_{F_3}$, and $\Lambda$ and $\mu$ coincide.
\end{lemma}
\begin{proof}
Reduction $\psi:R_k\to R_3$ is a ring homomorphism compatible with the residue maps. Apply $\psi$ to the
representatives in the definitions of $A_F$ and $E_F$. Then $\psi(2g)=2\psi(g)$ and $\psi(4c)=4\psi(c)$, and by
Lemmas~\ref{lem:A} and~\ref{lem:E} the residues do not depend on the representatives.
\end{proof}

\begin{theorem}\label{thm:CL}
Let $k\ge4$ and suppose $F_3$ minimizes $\Dexc$ at level $3$. Then $F_k$ minimizes $\Dexc$ at level $k$ if and
only if $A_{F_3}$ has no zero. In particular, if $A_{F_3}$ has no zero, every compatible lift of $F_3$ is
optimal at every level $k\ge4$.
\end{theorem}
\begin{proof}
By Lemma~\ref{lem:CL0} the invariants agree. Necessity is Theorem~\ref{thm:minimizers} for $k\ge4$. Conversely,
let $A$ have no zero. The level-$k$ criterion asks in addition, when $\mu\ne0$, that $\chi=1$ on $Z$. The level-$3$
criterion for $\mu\ne0$ gives $\chi(a)=1$ at every $a\in Z$ with $A(a)\ne0$, which is all of $Z$. When
$\mu\ne0$ and $Z=\emptyset$, a level-$3$ minimizer has $A\equiv0$, so both sides of the equivalence fail.
\end{proof}

\begin{corollary}\label{cor:CL2}
A compatible family $(F_k)_{k\ge3}$ minimizing $\Dexc$ at every level exists if and only if $\mu=0$ or
$Z\ne\emptyset$.
\end{corollary}
\begin{proof}
If $\mu=0$, take $F_3$ with $A\equiv1$ and apply Theorem~\ref{thm:CL}. If $\mu\ne0$ and $Z\ne\emptyset$, take
$F_3$ with $A\equiv1$ and $\chi\equiv1$ on $Z$, which minimizes at level $3$, and apply Theorem~\ref{thm:CL}.
If $\mu\ne0$ and $Z=\emptyset$, level $3$ requires $A\equiv0$ and level $k\ge4$ requires $A$ zero-free, for the
same function $A$ by Lemma~\ref{lem:CL0}.
\end{proof}

Theorem~\ref{thm:CL} and Corollary~\ref{cor:CL2} together give Corollary~\ref{cor:crosslevel}. The minimax table
gives the least value at each level. The two results above say when the same coefficients reach these values
at all levels at once, and we keep the two statements separate.

\begin{example}\label{ex:crosslevel}
For $m=5$ let $c_1$ be the lift with $A\equiv0$ of a representative with $\rho=28$, $\mu=17$ and $Z=\emptyset$. The same
integer coefficients give, over $\GR(8,5)$, $1024$ exceptional rows equal to ``$32$ appears $1024$ times''
(optimal value $q$). Over $\GR(16,5)$ every sampled exceptional row (four rows; the lower bound $q^2$ holds for all rows by
Theorem~\ref{thm:M1}(iii)) is ``$1024$ appears $1024$ times'', whereas the optimum is $64$. Lifts with $A$ zero-free behave as predicted by Theorem~\ref{thm:CL}.
\end{example}

\section{The family \texorpdfstring{$P_{m,r}$}{P(m,r)} and one-term optimal lifts}\label{sec:family}

Let $m\ge5$ be odd and $3\le r<m$ odd with $\gcd(r,m)=1$. Put $T_r=\sum_{i<r}X^{2^i}$ and, as
in~\cite[Thm.~3.1]{Fukui2},
\[
P_{m,r}=\sigma^{m-2}\circ X^3\circ T_r,\qquad\sigma(x)=x^2,
\]
reduced modulo $X^q-X$. With $u=2^{m-2}$, $v=2^{m-1}$, $w=2^{r-2}$ and $T=\sum_{i=0}^{r-3}X^{2^i}$ its reduced
form is
\[
p=T+X^w+X^wT+X^{w+u}+X^{w+v}+X^uT+X^{u+v},
\]
and for $r=3$ it is $p=X+X^2+X^3+X^{1+u}+X^{2+u}+X^{2+v}+X^{u+v}$. The natural lift $F$ reads the same
expression with integer coefficients $1$. Its degree is $u+v=3q/4$.

\begin{proposition}[The residue is APN]\label{prop:resAPN}
Put $Y(x)=T(x)+x^u+x^v$. On $\F_q$ we have $p(x)=Y(x)^3$, the map $Y$ is additive, and $Y(x)=0$ only for
$x=0$. Hence $p$ is an APN permutation.
\end{proposition}
\begin{proof}
On $\F_q$ we have $(x^u)^2=x^v$ and $(x^v)^2=x$, and $T^2=T+x+x^w$. Expanding $(T+x^u+x^v)^3$ and using these
relations in characteristic $2$ gives the reduced form of $p$. Additivity is clear. Further
$Y^4=T^4+x+x^2$, and $T^4+x+x^2=T_r(x)$. If $T_r(x)=0$ then $T_r+T_r^2=x+x^{2^r}$ gives $x^{2^r}=x$, so
$x\in\F_{2^{\gcd(r,m)}}=\F_2$. Since $T_r(1)=r=1$, we get $x=0$. For $c\ne0$ and $d=Y(c)\ne0$,
$p(t+c)-p(t)=dY(t)^2+d^2Y(t)+d^3$. Two solutions $t,t_0$ of $p(t+c)-p(t)=b$ satisfy
$d(Y(t)-Y(t_0))(Y(t)+Y(t_0)+d)=0$, so $t\in\{t_0,t_0+c\}$. Bijectivity of $p$ follows from that of $Y$
and of $x\mapsto x^3$ for odd $m$.
\end{proof}

This is the route of~\cite[Thm.~3.1]{Fukui2} (the cube is APN, $T_r$ is invertible, and the outer map is a
power of Frobenius), with $\sigma^{m-2}$ absorbed into $Y$. The inner power map is the cube; $r$ only indexes the inner linear map
$T_r$ (an earlier version of our notes wrongly described $P_{m,r}$ as linearly equivalent to $X^{2^r+1}$).

\begin{proposition}\label{prop:I123}
(i) $\rho=1$: $p'=1+X^w+X^u$ for $r\ge5$ and $p'=1+X^2+X^u$ for $r=3$, with $\ker K=\{0,1\}$ and $1\notin K(\F_q)$.
(ii) $\mu=0$ for $r\ge5$ and $\mu=1$ for $r=3$.
(iii) $\Lambda\equiv0$, so $Z=\F_q$.
\end{proposition}
\begin{proof}
(i) Only monomials with odd exponent contribute to $p'$. Further, $x^w=x^u$ implies $x^{2^{m-r}}=x$ after applying Frobenius,
so $x\in\F_{2^{\gcd(r,m)}}=\F_2$. Also $K(\F_q)\subseteq\ker\Tr$ while $\Tr(1)=m=1$.
(ii) By Lemma~\ref{lem:F0}(i) only the exponents $2$, $3$ and $2+2^j$ contribute to $p^{[2]}$. This gives
$p^{[2]}=1+X^w+X^u$ for $r\ge5$ and $1+X+X^u+X^v$ for $r=3$.
(iii) In characteristic $2$, $D_1X^{2^i}=1$ and $D_1X^{2^a+2^b}=X^{2^a}+X^{2^b}+1$. Summing these over the
monomials of $p$ and adding $p'$ gives $0$ (for $r\ge5$ the number of terms of $T$ is $r-2$, which is odd).
\end{proof}

\begin{proposition}\label{prop:I45}
Modulo $4$ and on $\F_q$:
(i) for $r\ge5$, $A_F=L+L^2$ with $L=X^{w/2}+X^{u/2}$, and $A_{F+2X^3}=L+L^2+1$, which has no zero;
(ii) for $r=3$, $A_F=1+X+X^{u/2}+X^u+X^v$, which has no zero, and $A$ is unchanged by $2X^{4+v}$.
\end{proposition}
\begin{proof}
We use $(X+1)^{2^j}\equiv X^{2^j}+1+2X^{2^{j-1}}\pmod 4$ for $j\ge1$ (induction on $j$). For $r\ge5$,
$F'\equiv1+2X+X^w+2X^{w+1}+X^u+2X^{u+1}\pmod4$, and substituting $X+1$ gives the difference
$2(X^{w/2}+X^w+X^{u/2}+X^u)$. The term $2X^3$ adds $D_1(X^2)=1$ by Proposition~\ref{prop:Acontrol}(i). Since
$\Tr(L+L^2)=0$ and $\Tr(1)=1$, the function $L+L^2+1$ has no zero. For $r=3$ the same computation gives
$A_F=1+X+X^{u/2}+X^u+X^v$, whose trace is $1+4\Tr(X)=1$. The correction $2J^2$ with $J=X^{2+u}$ keeps $A$ by
Proposition~\ref{prop:Econtrol}.
\end{proof}

\begin{proposition}\label{prop:I678}
Let $r=3$, $Y_0=X^{u/2}$, $U=X^u$, $V=X^v$.
(i) $E_F=X+X^2+Y_0+XY_0+X^2Y_0+X^2U+V+Y_0V$.
(ii) For $F^\sharp=F+2X^{4+v}$, $E_{F^\sharp}=E_F+X^4+V+1$. With $\varphi=1+X^2+XV+Y_0V+UV$ and
$A=1+X+Y_0+U+V$ we have $\varphi^2+A\varphi=E_{F^\sharp}+A^2$ on $\F_q$. Hence $\chi_{F^\sharp}\equiv1$.
(iii) $E_F(0)=0$ and $A_F(0)=1$, so $\chi_F(0)=0$, and the natural lift has $\nu\ge2$.
\end{proposition}
\begin{proof}
(i) Modulo $8$, $S(x)\equiv2[F(x)-F(x+1)+F'(x)]+4F^{[2]}(x)$, so
\[E_F=\ol{(F(X+1)-F(X)-F'(X))/2}+p^{[2]}.\]
The numerator is even because $\Lambda\equiv0$, and expanding it
modulo $4$ gives the stated form. (ii) $D_1(X^{2+u})=X^2+U+1$, so $(D_1j)^2=X^4+V+1$ by
Proposition~\ref{prop:Econtrol}. The identity is an identity in $\F_2[t]$ after the substitution $Y_0=t$,
$U=t^2$, $V=t^4$, $X=t^8$, which is valid on $\F_q$ because $V^2=X$. It gives
\[\Tr(\mu E_{F^\sharp}/A^2)=\Tr\bigl(1+(\varphi/A)^2+\varphi/A\bigr)=1.\]
(iii) Direct evaluation, and $\chi$ is $\rho$-periodic
(Lemma~\ref{lem:overlap}).
\end{proof}

\begin{proof}[Proof of Theorem~\ref{thm:oneterm}]
The correction is divisible by $2$, so the residue is unchanged and $F^\sharp$ is a permutation. The correction
has degree $3$ or $q/2+4$, both below $3q/4$, so $F^\sharp$ keeps the degree $3q/4$ of $F$. For $r\ge5$: $\mu=0$, $Z=\F_q$, and $A_{F^\sharp}$ has no zero
(Propositions~\ref{prop:I123} and~\ref{prop:I45}). Theorem~\ref{thm:S3} gives $n_{2q}=q^{k-1}/2$ and
$n_q=0$. For $r=3$: $\mu=1$, $Z=\F_q$, $A$ has no zero and $\chi\equiv1$ (Proposition~\ref{prop:I678}),
so $\nu=0$ and $n_{2q}=q^{k-1}/2$. Ordinary directions are covered by Lemma~\ref{lem:ordinary} and
Proposition~\ref{prop:resAPN}. Optimality is the row $Z\ne\emptyset$ of Theorem~\ref{thm:minimax}. For the
natural lift with $r=3$, $A_F$ has no zero, so Theorem~\ref{thm:S3} bounds every entry by $4q$, and
Proposition~\ref{prop:I678}(iii) gives $n_{4q}>0$, so $\Dexc=4q$. For $r\ge5$,
$A_F=L+L^2$ vanishes where $L\in\F_2$. Theorem~\ref{thm:M1}(iii) and the case $k=3$, $\mu=0$ in the proof of
Theorem~\ref{thm:minimizers} give $\Dexc\ge q^2$.

Cost: we count operations in the \emph{termwise evaluation model}, in which the squares $x^{2^i}$ ($1\le i\le m-1$)
are computed once, each monomial is formed from them by ring multiplications, multiplied by its coefficient when
the coefficient is not $1$, and the terms are added. The squares are already computed for $F$. For $r\ge5$,
$x^3=x\cdot x^2$; for $r=3$, $x^{4+v}=x^4\cdot x^{2^{m-1}}$. Each costs one ring multiplication. Doubling and adding cost
two additions. Other evaluation circuits of the same polynomial (factored forms, shared products, hardware) have their
own costs, which this count does not address.
\end{proof}

\begin{corollary}[The natural lift at $k=3$]\label{cor:natural}
Let $r\ge5$, let $F$ be the natural lift of $P_{m,r}$ over $R_3=\GR(8,m)$ and fix an exceptional $\alpha$. Then
$n_{2q}=q(q-2)/2$, and either $n_{q^2}=2$ (if $r\equiv1\pmod4$) or $n_{2q^2}=1$ (if $r\equiv3\pmod4$). All other
positive values have multiplicity zero. In particular $\Dexc(F)=q^2$ if $r\equiv1\pmod4$ and $\Dexc(F)=2q^2$ if
$r\equiv3\pmod4$.
\end{corollary}
\begin{proof}
By Propositions~\ref{prop:I123} and~\ref{prop:I45}, $\rho=1$, $\mu=0$, $Z=\F_q$ and $A_F=L+L^2$ with
$L=X^{w/2}+X^{u/2}$. If $L(a)=0$, raising $a^{2^{r-3}}=a^{2^{m-3}}$ to the power $8$ gives $a^{2^r}=a^{2^m}=a$, so
$a\in\F_{2^{\gcd(r,m)}}=\F_2$. The value $L(a)=1$ is impossible, since $\Tr L(a)=\Tr(a)+\Tr(a)=0$ while $\Tr(1)=1$. Hence
$A_F(a)=0$ exactly for $a\in\{0,1\}$.

Fix a pair $\{a,a+1\}$ and $\beta$ with $\ol\beta=D_1p(a)$; only this pair contributes (Lemma~\ref{lem:F0}(iii)).
Since $\Lambda\equiv0$, both classes pass the gate for the same $\beta\bmod4$ (Lemma~\ref{lem:gate}), and
$c_{a+1}=c_a+E_F(a)$ (Lemma~\ref{lem:local}(iii)). If $a\notin\{0,1\}$, then $L_a(y)=A_F(a)y$ is bijective, and by
Lemma~\ref{lem:local}(ii) each class contributes $q$ for each of the $q$ values of $c_a$. This gives $q$ values of
$\beta$ with entry $2q$ for each of the $(q-2)/2$ such pairs. For the pair $\{0,1\}$, $L_0\equiv0$, so the class of
$a$ contributes $q\cdot q=q^2$ if $c_a=0$ and nothing otherwise. With the lifts $x=0$, $\tau=1$ in Lemma~\ref{lem:E},
$E_F(0)$ is the residue of $(F(2)-2F(1)+F(0))/4$. The natural lift has $3r-2$ monomials, one with exponent $1$, one
with exponent $2$ and all others with exponent at least $4$, so $F(1)=3r-2$, $F(2)\equiv6\pmod8$ and $F(0)=0$. Thus
$F(2)-2F(1)+F(0)\equiv10-6r\pmod8$, and $E_F(0)=(5-3r)/2\bmod2$, which is $1$ for $r\equiv1\pmod4$ and $0$ for
$r\equiv3\pmod4$. In the first case $c_0=0$ and $c_0=1$ each give one entry $q^2$; in the second case $c_0=0$ gives one
entry $2q^2$. As a check, $2q\cdot q(q-2)/2+2q^2=q^3$.
\end{proof}

Corollary~\ref{cor:natural}, multiplicities included, is formalized (Appendix~\ref{app:verification}). It agrees
with the full computation of all $16384$ exceptional rows of the natural lift of $P_{7,5}$
($n_{2q}=8064$, $n_{q^2}=2$). For $r\equiv3\pmod4$ we checked sampled rows of the natural lift of $P_{9,7}$ over
$\GR(8,9)$ (Appendix~\ref{app:verification}).

For $r\ge5$ the lift $F^\sharp$ also optimizes the $k=3$ middle layer and makes the penultimate layer optimally flat
(Theorem~\ref{thm:layers}, $\mu=0$). For $r=3$ we have $\mu=1$, so an optimal middle layer
would need $A\equiv0$, whereas $A_{F^\sharp}$ has no zero (Theorem~\ref{thm:minimizers}).

\section{The lower layers}\label{sec:layers}

\begin{theorem}[Middle layer, $k=3$]\label{thm:O3}
Let $k=3$ and $\alpha=2u$ with $u$ a unit. Then $F(x+\alpha)-F(x)$ depends only on $\ol x$, and residues $a$ and $b$ give the same value
if and only if $b=a$, or $b=a+\rho$ and $\ol u A_F(a)+\ol u^2\mu=0$. Hence all nonzero entries of every row in
$v_1$ equal $q^2$ if and only if $A_F\equiv0$ when $\mu\ne0$, and $A_F$ has no zero when $\mu=0$.
\end{theorem}
\begin{proof}
$F(x+2u)-F(x)=2uF'(x)+4u^2F^{[2]}(x)+8(\cdots)$. Both terms depend only on $\ol x$ modulo $8$
(Lemma~\ref{lem:A}(i)). Modulo $4$ the value is the lift of $2\ol up'(a)$, so equal values force
$b\in\{a,a+\rho\}$. For $b=a+\rho$ and $x_b=x_a+\tau$ the difference is $\equiv4(\ol uA_F(a)+\ol u^2\mu)\pmod 8$.
As $\ol u$ runs over $\F_q^*$, the condition ``never zero'' means $A=0$ if $\mu\ne0$ and $A\ne0$ if $\mu=0$.
\end{proof}

\begin{theorem}[Penultimate layer]\label{thm:M4}
Let $k\ge4$ and $\alpha\in2^{k-2}R_k\setminus2^{k-1}R_k$. Every row has $q-s$ entries $q^{k-1}$ and $s/2$ entries
$2q^{k-1}$. In particular the penultimate layer is optimally flat (all nonzero entries equal to $q^{k-1}$) if and
only if $A_F$ has no zero.
\end{theorem}
\begin{proof}
$\alpha^2=0$, so $F(x+\alpha)-F(x)=\alpha F'(x)$, which depends only on $F'(x)\bmod4$ and hence only on $\ol x$.
Two residues give the same value if and only if $F'(x_a)\equiv F'(x_b)\pmod4$. This forces $b\in\{a,a+\rho\}$,
and for $b=a+\rho$ it holds if and only if $A_F(a)=0$.
\end{proof}

\begin{corollary}\label{cor:M5}
(i) If $\mu\ne0$, middle-layer optimality at $k=3$ ($A\equiv0$) and exceptional optimality at $k\ge4$ ($A$
zero-free) are incompatible. A lift with $A\equiv0$ has $\Dexc\ge q^2$ and a penultimate layer whose nonzero entries
all equal $2q^{k-1}$, at every $k\ge4$.
(ii) If $\mu=0$, a zero-free $A$ gives the middle layer at $k=3$, exceptional optimality at all $k$ and an
optimally flat penultimate layer at all $k\ge4$.
(iii) If $\mu\ne0$ and $Z=\emptyset$, the optimal exceptional value itself rises from $q$ at $k=3$ to $2q$ at
$k\ge4$. This change of value is a separate phenomenon from the incompatibility in (i).
\end{corollary}

This proves Theorem~\ref{thm:layers}.

\paragraph{Intermediate layers.}
For $1\le j\le k-3$ the layers $v_j$ lie between the middle layer of $k=3$ and the penultimate layer. Two further
invariants of $p$ enter: $c_3=[X^3]p$ and the affine map $\pi=p^{[2]}$ of Lemma~\ref{lem:F0}(i). For $\alpha=2^ju$
with $u$ a unit put
\begin{gather*}
w_\alpha=\pi+[j=1]\,\ol u\,c_3,\qquad B_\alpha=A_F+[j=1]\,\ol u\,\mu,\\
s_\alpha=\#B_\alpha^{-1}(0),\qquad d=k-j-2 .
\end{gather*}

\begin{theorem}[Intermediate layers]\label{thm:inter}
Let $k\ge4$, $1\le j\le k-3$ and $\alpha=2^ju$ with $u$ a unit.
\begin{enumerate}[label=(\roman*)]
\item Every nonzero entry of the row of $\alpha$ is a multiple of $q^{j+1}$.
\item If $c_3=0$ and $\pi$ has no zero, the row has $(q-s_\alpha)q^{d}$ entries $q^{j+1}$, $\frac{s_\alpha}2q^{d}$ entries
$2q^{j+1}$, and no other nonzero entries.
\item All nonzero entries of all rows of $v_j$ equal $q^{j+1}$ if and only if $c_3=0$, $\pi$ has no zero, and $A_F$ has
no zero (when $j\ge2$ or $\mu=0$), respectively $A_F\equiv0$ (when $j=1$ and $\mu\ne0$).
\end{enumerate}
\end{theorem}
\begin{proof}
Put $n=k-j$. Then $F(x+\alpha)-F(x)=\alpha T(x)$ with $T=\sum_{i\ge1}\alpha^{i-1}F^{[i]}$, and $\alpha y=\alpha y'$ if and
only if $y\equiv y'\pmod{2^n}$. Hence $\delta_F(\alpha,\beta)=0$ unless $\beta=\alpha\gamma$, and then
$\delta_F(\alpha,\alpha\gamma)=q^j\,\#\{y\in R_k/2^nR_k:T(y)\equiv\gamma\}$. Since $\ol T=p'$ is two-to-one on the pairs
$\{a,a+\rho\}$, solutions lie in one pair. The derivative $T'=2F^{[2]}+3\alpha F^{[3]}+\cdots$ has coefficients in
$2R_k$; with $W=T'/2$ we get $T(x+2z)=T(x)+4\Phi_x(z)$, where $\Phi_x(z)=W(x)z+T^{[2]}(x)z^2+\sum_{r\ge3}2^{r-2}T^{[r]}(x)z^r$
has residue $w_\alpha(a)y+c_3y^2$ for $\ol x=a$ (the residue of $W$ is $\pi+[j=1]\ol uc_3$ and that of $T^{[2]}$ is
$p^{[3]}=c_3$; the only exponent $2^a+2^b$ with $\binom e3$ odd is $3$). On the class of $a$, $T(x_a+2z)=\gamma$ means
$\gamma\equiv T(x_a)\pmod4$ and $\Phi_{x_a}(z)\equiv g\pmod{2^{d}}$ with $g=(\gamma-T(x_a))/4$; the second condition
depends on $z$ modulo $2^d$, and each such $z$ has $q$ lifts modulo $2^{n-1}$. So the class contributes
$q^{j+1}N_a(g)$ with $N_a(g)=\#\{z\bmod2^d:\Phi_{x_a}(z)\equiv g\}$, which proves (i).
Moreover $T(x+\tau)-T(x)\equiv2[A_F(a)]+[j=1]\,2[\ol u\mu]=2[B_\alpha(a)]\pmod4$, so the values of $\gamma$
reached from the two classes of a pair lie in the same residue class modulo $4$ if $B_\alpha(a)=0$ and in different
residue classes modulo $4$ otherwise; in the second case the two classes reach disjoint sets of values.

(ii) Here $w_\alpha=\pi$ has no zero and $c_3=0$, so $\Phi_{x_a}$ has unit derivative and every $g$ has exactly one
solution by Hensel's lemma: $N_a\equiv1$ on all $q^d$ values of $g$. So each class of a pair reaches every $\gamma$ of
its residue class modulo $4$, each with the entry $q^{j+1}$. A pair with $B_\alpha\ne0$ therefore gives $2q^d$
entries $q^{j+1}$. For a pair with $B_\alpha=0$ both classes reach every $\gamma$ of the same residue class modulo $4$, so the
pair gives $q^d$ entries $2q^{j+1}$.

(iii) If $c_3\ne0$, then $w_\alpha$ is affine with nonzero linear part, so $w_\alpha(a)\ne0$ for some $a$, and
$\#\{y:w_\alpha(a)y+c_3y^2=\ol g\}\in\{0,2\}$ gives entries $2q^{j+1}$ or more. If $c_3=0$ and $\pi(a)=0$, then
$\Phi_{x_a}\equiv0\pmod2$, so its $q^d$ values fall into at most $q^{d-1}$ classes and some $N_a(g)\ge q$. Otherwise (ii)
applies, and the row is flat if and only if $s_\alpha=0$. For $j\ge2$ or $\mu=0$, $B_\alpha=A_F$. For $j=1$ and
$\mu\ne0$, $B_\alpha$ has no zero for every unit $\ol u$ if and only if $A_F(a)/\mu\notin\F_q^*$ for all $a$, that is
$A_F\equiv0$.
\end{proof}

\begin{corollary}\label{cor:interfamily}
Let $r\ge5$ in Theorem~\ref{thm:oneterm}. For every $k\ge3$ and every $1\le j\le k-2$, every row of $v_j$ of $F^\sharp$
has all nonzero entries equal to $q^{j+1}$, the least possible value. The natural lift has, in every row of $v_j$,
$(q-2)q^{k-j-2}$ entries $q^{j+1}$ and $q^{k-j-2}$ entries $2q^{j+1}$.
\end{corollary}
\begin{proof}
For $r\ge5$, $p$ has no monomial $X^3$, so $c_3=0$, and $\pi=1+X^w+X^u$ (Proposition~\ref{prop:I123}(ii)) has no zero
because $\Tr\pi(a)=\Tr(1)=1$. Also $\mu=0$, $A_{F^\sharp}$ has no zero and $A_F^{-1}(0)=\{0,1\}$
(Corollary~\ref{cor:natural}). Apply Theorem~\ref{thm:inter} for $j\le k-3$, Theorem~\ref{thm:M4} for $j=k-2$
and Theorem~\ref{thm:O3} for $k=3$, $j=1$. Minimality is Theorem~\ref{thm:inter}(i) for $j\le k-3$, Theorem~\ref{thm:M4} for $j=k-2$,
and Theorem~\ref{thm:O3} for $k=3$, where the difference is constant on residue classes of $q^2$ points.
\end{proof}

Thus for $r\ge5$ the one-term lift is optimal on every layer except the top layer, whose value $2q^{k-1}$ is forced
for every lift~\cite{RS}. When $c_3\ne0$ or $\pi$ has a zero, the class counts $N_a$ are given by a Hensel tree whose
shape depends on whether $w_\alpha(a)=0$ and whether $c_3=0$. A residue with $w_\alpha(a)=0\ne c_3$ produces a
universal multiset of values, from which every row of $v_j$ meeting such a residue has an entry at least
$q^{\,j+1+\lfloor (k-j-2)/2\rfloor}$. For $r=3$ ($c_3=1$, $\pi(1)=0$, and $X+X^u+X^v$ is bijective because $3\nmid m$) this happens for every lift.

The finer multisets in these cases are governed by one elementary recursion. For $\Phi\in R_k[z]$ and $D\ge0$ let
$N^D_\Phi(g)=\#\{z\in R_k/2^DR_k:\Phi(z)\equiv g\pmod{2^D}\}$; the class contribution in the proof of
Theorem~\ref{thm:inter} is $q^{j+1}N^d_{\Phi_{x_a}}$.

\begin{proposition}[Digit recursion]\label{prop:GR}
Let $\Phi\in R_k[z]$, $2\le D\le k$ and $\hat y\in R_k$ a lift of $y\in\F_q$. If $\Phi'(\hat y)$ is a unit, then for every
$g$ with $\ol g=\ol\Phi(y)$ exactly one $z\equiv\hat y\pmod2$, $z\bmod 2^D$, satisfies $\Phi(z)\equiv g$. If
$\Phi'(\hat y)\in2R_k$, then $\Phi_y(z)=(\Phi(\hat y+2z)-\Phi(\hat y))/4$ lies in $R_k[z]$, and
\[
\#\{z\equiv\hat y\ (2):\Phi(z)\equiv g\ (2^D)\}=q\,[\,g\equiv\Phi(\hat y)\ (4)\,]\;N^{D-2}_{\Phi_y}\bigl((g-\Phi(\hat y))/4\bigr),
\]
for any choice of the quotients by $4$; the count does not depend on the choice since $D-2\le k-2$.
\end{proposition}
\begin{proof}
If $\Phi'(\hat y)$ is a unit, $\Phi(z)-\Phi(z')=(z-z')U$ with $U\equiv\Phi'(\hat y)\pmod2$ for $z\equiv z'\equiv\hat y$, so
$\Phi$ is injective from $\hat y+2R_k/2^D$ into $\Phi(\hat y)+2R_k/2^D$, two sets of $q^{D-1}$ elements. Otherwise
$\Phi(\hat y+2z)-\Phi(\hat y)=\sum_{r\ge1}2^r\Phi^{[r]}(\hat y)z^r$ has all coefficients in $4R_k$, and
$\Phi(\hat y+2z)\equiv g\pmod{2^D}$ holds if and only if $g\equiv\Phi(\hat y)\pmod4$ and $\Phi_y(z)\equiv(g-\Phi(\hat y))/4
\pmod{2^{D-2}}$, a condition on $z$ modulo $2^{D-2}$; each such $z$ has $q$ lifts modulo $2^{D-1}$.
\end{proof}

Iterating Proposition~\ref{prop:GR} over the residue classes determines every row of every intermediate layer of every
lift, after at most $\lceil d/2\rceil$ steps and with $O(q)$ evaluations of $T$ per critical point. The notes on the
intermediate layers in Online Resource~1 classify the resulting multisets. Every configuration of a single class or a merged
pair is given in closed form in terms of level data of $T$, except a few configurations of merged or deep classes, listed in those
notes, that need at least five independent conditions on the Taylor data of $F$; there the recursion
still gives the answer. In particular, for $j\le k-4$, a class with $c_3=0=\pi(a)$ for which the coefficients of $z$
and $z^2$ in $\Phi_{x_a}$ lie in $4R_k$ has an entry at least $q^{j+3}$.
All closed forms and the recursion agree with exhaustive computation for $m=5$ and $d\le5$ ($133$ configurations, $75$
full rows, and $691$ further rows from earlier computations).

\section{Cryptographic consequences}\label{sec:crypto}

Write $P_F(\alpha\to\beta)=\delta_F(\alpha,\beta)/q^k$. For $F^\sharp$ the maximum over unit $\alpha$ is
$2q/q^k=2/q^{k-1}$, and over ordinary $\alpha$ at most $2/q^k$. For a general $p$ and an optimal lift the
exceptional maximum is the minimax value divided by $q^k$.

\begin{proposition}[Valuation preservation]\label{prop:valuation}
If $p$ satisfies (H) and $F$ is a lift of $p$, then $v_2(F(x+\alpha)-F(x))=v_2(\alpha)$ for every $\alpha\ne0$ and every $x$.
\end{proposition}
\begin{proof}
For a unit $\alpha$ the residue difference is nonzero since $p$ is a permutation. For $\alpha=2^su$ with
$s\ge1$, $F(x+\alpha)-F(x)=\alpha(F'(x)+\alpha H(x,\alpha))$, and the bracket is a unit since $p'$ has no zero.
\end{proof}

\paragraph{The probability space.}
Let the state be $x_0=X_0\in R_k^n$, uniform and independent of independent uniform round keys
$\kappa_1,\dots,\kappa_N\in R_k^n$, and $x_i=M\cdot S(x_{i-1})+\kappa_i$. Equivalently, take a fixed plaintext
and add an independent uniform whitening key $\kappa_0$. For a characteristic
$\Omega=(\Delta_0\to\dots\to\Delta_N)$ let $\Pr[\Omega]$ be the probability, over $(X_0,\kappa_1,\dots,\kappa_N)$,
that the pair $(x_0,x_0+\Delta_0)$ follows $\Omega$.

\begin{example}[A fixed plaintext breaks the product formula]\label{ex:fixed}
Take $\GR(8,5)$, $S=F^\sharp$ for $P_{5,3}$, $n=1$, $M=1$ and the fixed pair $(0,1)$. Then $S(1)-S(0)=1$, so the first
transition holds for every key. The second input $x_1=S(0)+\kappa_1$ is uniform, and the key-averaged
probability of $1\to1\to1$ is $\delta(1,1)/|R|=64/32768=1/512$. The product of transition probabilities is
$(1/512)^2$.
\end{example}

Proposition~\ref{prop:markov} and part (a) of Corollary~\ref{cor:crypto} are the Markov-cipher argument of Lai,
Massey and Murphy~\cite{LMM} and the branch-number argument of Daemen and Rijmen~\cite{DRwide,DR}, written for keys
added in $R_k$ and combined with the valuation preservation above, and Proposition~\ref{prop:hull} below is the
two-round bound of Hong et al.~\cite{HLLSCC} (see also~\cite{PSLL}) transferred to ring addition. What is new is the
per-S-box input: the sharp unit-direction value of Theorems~\ref{thm:oneterm} and~\ref{thm:minimax}, the layer values of
Section~\ref{sec:layers}, the transfer of the branch number of $\ol M$ to $R_k^n$, and, in
Proposition~\ref{prop:hullunit}, the observation that valuation preservation fixes the active words of a layer at both
ends of a two-round differential, so that the layer values of a lift enter the differential bound without a
combinatorial factor. Throughout, only the additivity of $M$ and the invertibility of $\ol M$ are used.

\begin{proposition}[Markov property]\label{prop:markov}
In the probability space above, $\Pr[\Omega]=\prod_{i=1}^N P_S(\Delta_{i-1}\to M^{-1}\Delta_i)$, where
\[
P_S(a\to b)=\prod_{j=1}^{n}\frac{\delta_F(a_j,b_j)}{q^k}=q^{-kn}\prod_{j=1}^{n}\delta_F(a_j,b_j)
\]
is the probability that a uniform $x\in R_k^n$ satisfies $S(x+a)-S(x)=b$.
\end{proposition}
\begin{proof}
Let $E_i$ be the event $S(x_{i-1}+\Delta_{i-1})-S(x_{i-1})=M^{-1}\Delta_i$. The events $E_1,\dots,E_{i-1}$ are
determined by $(X_0,\kappa_1,\dots,\kappa_{i-2})$. The input $x_0=X_0$ is uniform, and for $i\ge2$,
$x_{i-1}=MS(x_{i-2})+\kappa_{i-1}$ with $\kappa_{i-1}$ uniform and independent of
$(X_0,\kappa_1,\dots,\kappa_{i-2})$. So $x_{i-1}$ is uniform conditionally on $E_1\cap\dots\cap E_{i-1}$. Hence
$\Pr[E_i\mid E_1\cap\dots\cap E_{i-1}]=P_S(\Delta_{i-1}\to M^{-1}\Delta_i)$.
\end{proof}

\begin{proof}[Proof of Corollary~\ref{cor:crypto}]
Let $\Pr[\Omega]>0$ and let $\ol\Delta_i$ and $u_i$ be the reductions modulo $2$ of $\Delta_i$ and of
$M^{-1}\Delta_i$. By Proposition~\ref{prop:valuation}, $\mathrm{supp}(u_i)=\mathrm{supp}(\ol\Delta_{i-1})$, and
$\ol\Delta_i=\ol Mu_i$. With $a_i=\mathrm{wt}(\ol\Delta_{i-1})$, the number of unit-active S-boxes in round $i$, we get
$a_i+a_{i+1}=\mathrm{wt}(u_i)+\mathrm{wt}(\ol Mu_i)\ge B$. Since $\ol\Delta_0\ne0$, all $u_i\ne0$, and
$\sum_ia_i\ge B\lfloor N/2\rfloor$. Each unit-active S-box contributes at most $2/q^{k-1}$, and the others at most $1$.
This proves (a). Part (b) is Proposition~\ref{prop:layerchar}, part (c) is Propositions~\ref{prop:hull}
and~\ref{prop:hullunit}, with $D_0=2q$ from Theorem~\ref{thm:oneterm} and $D_s$ from Proposition~\ref{prop:layerchar}, and part (d) is
Proposition~\ref{prop:fixedkey}.
\end{proof}

For example, with $m=7$, $k=3$ (a $21$-bit S-box), $n=4$, an MDS matrix $\ol M$ ($B=5$) and $N=8$, this bound
is $2^{-260}$ for $F^\sharp$ of $P_{7,5}$. For the natural lift of $P_{7,5}$ the worst exceptional entry at
$k=3$ is exactly $q^2$ (Corollary~\ref{cor:natural}; all $16384$ exceptional rows were also computed,
Appendix~\ref{app:verification}), so the same argument gives only $2^{-140}$, and no better bound per unit-active S-box
is possible with this method.

\paragraph{Lower layers.}
Every nonzero difference $\Delta\in R_k^n$ has a layer: the largest $s$ with $\Delta\in(2^sR_k)^n$. Then $s\le k-1$
and some word of $\Delta$ lies outside $2^{s+1}R_k$. For $s<k$ the digit $\langle2^sy\rangle_s=\ol y$ of an element of
$2^sR_k$ is well defined, since $2^s(y-y')=0$ forces $y-y'\in2^{k-s}R_k\subseteq2R_k$. The branch-number argument works
in every layer.

\begin{proposition}[Characteristics in every layer]\label{prop:layerchar}
In the setting of Corollary~\ref{cor:crypto}, let $F$ be any lift of $p$ and $1\le s\le k-1$, let the input difference
$\Delta_0$ have layer $s$, and let $D_s$ be the largest entry of the rows $\delta_F(2^su,\cdot)$, $u$ a unit. Then
$\Pr[\Omega]\le(D_s/q^k)^{B\lfloor N/2\rfloor}$. For every lift $D_{k-1}=2q^{k-1}$, and for $F^\sharp$ with $r\ge5$,
$D_s=q^{s+1}$ for $1\le s\le k-2$.
\end{proposition}
\begin{proof}
Let $\Pr[\Omega]>0$. A polynomial maps a difference in $2^sR_k$ to a difference in $2^sR_k$, since $x-y$ divides
$F(x)-F(y)$, and $M$ maps $(2^sR_k)^n$ into itself because it is additive. So every $\Delta_i$ and every S-box output
difference lies in $(2^sR_k)^n$. By Proposition~\ref{prop:valuation} a word of an S-box input difference has nonzero
digit if and only if the corresponding output word has, and $\langle M(2^sy)\rangle_s=\ol M\,\ol y$ because
$M(2^sy)=2^sM(y)$. Since $\ol M$ is invertible, the digit vectors stay nonzero, so the layer $s$ is invariant along
$\Omega$. With $a_i$ the number of words of $\Delta_{i-1}$ with nonzero digit we get $a_i+a_{i+1}\ge B$ exactly as in
the proof of Corollary~\ref{cor:crypto}(a), and each such word contributes at most $D_s/q^k$. For every lift,
$F(x+2^{k-1}u)-F(x)=2^{k-1}uF'(x)$ depends on $\ol x$ only through $\ol u\,p'(\ol x)$, and $p'$ is two-to-one, so every
entry of a top-layer row is $0$ or $2q^{k-1}$. For $F^\sharp$ and $s\le k-2$ the value is
Corollary~\ref{cor:interfamily}.
\end{proof}

In the example above ($m=7$, $k=3$, $n=4$, $B=5$, $N=8$) the layer $s=1$ gives $2^{-140}$ for $F^\sharp$ but only
$2^{-120}$ for the natural lift, whose middle-layer rows contain the value $2q^2$ because $A_F$ vanishes at $0$ and $1$
(Theorem~\ref{thm:O3}). The top layer $s=2$ gives $2^{-120}$ for every lift. So, for $r\ge5$, the one-term lift improves
the bound in every layer below the top, and the weakest of these bounds is the top-layer one, which is the same for every
lift.

\paragraph{Differentials and fixed keys.}
A differential is the union of all characteristics with given input and output differences. For $x\in R_k^n$ let
$\mathrm{wt}(x)$ be the number of nonzero words.

\begin{proposition}[Two-round differentials]\label{prop:hull}
Let $M$ be an additive automorphism of $R_k^n$ whose reduction $\ol M$ has branch number $B$, let $S$ apply a
permutation $G$ of $R_k$ to each word, and let $\eta=\max\{\delta_G(\alpha,\beta)/q^k:\alpha\ne0\}$. For every $a\ne0$
and every $b$, the probability, over a uniform $x_0$ and independent uniform $\kappa_1,\kappa_2$, that the pair
$(x_0,x_0+a)$ has output difference $b$ after the rounds $x\mapsto M\cdot S(x)+\kappa_i$, $i=1,2$, is at most
$\eta^{B-1}$. The same bound holds after any number $N\ge2$ of rounds with independent uniform keys.
\end{proposition}
\begin{proof}
First, $\mathrm{wt}(x)+\mathrm{wt}(Mx)\ge B$ for every nonzero $x\in R_k^n$: if $s$ is the layer of $x$ and $x=2^sy$,
then $\ol y\ne0$, $\mathrm{wt}(x)\ge\mathrm{wt}(\ol y)$ and $\mathrm{wt}(Mx)\ge\mathrm{wt}(\ol M\ol y)$, because
$2^sv\ne0$ for a unit $v$. Second, summing over the difference $x$ after the first S-box layer, the key $\kappa_1$
makes the second input uniform, so the probability is $\sum_xP_S(a\to x)\,P_S(Mx\to c)$ with $c=M^{-1}b$ and $P_S$
as in Proposition~\ref{prop:markov}. Since $G$ is a bijection, a term is nonzero only if $x$ and $Mx$ have the
supports $I$ of $a$ and $J$ of $c$. If some term is nonzero, then $|I|+|J|\ge B$; choose $I'\subseteq I$ and
$J'\subseteq J$ with $|I'|+|J'|=B-1$ and bound the factors indexed by $I'$ and $J'$ by $\eta$. Two contributing $x$
with the same coordinates on $I\setminus I'$, whose images under $M$ have the same coordinates on $J\setminus J'$,
differ by some $e$ with $\mathrm{wt}(e)+\mathrm{wt}(Me)\le B-1$, so $e=0$. The remaining sum is therefore at most
$\prod_{j\in I\setminus I'}\sum_v\delta_G(a_j,v)/q^k\cdot\prod_{j\in J\setminus J'}\sum_u\delta_G(u,c_j)/q^k=1$, since
the rows and the columns of the difference table of a permutation sum to $q^k$. For $N\ge2$ rounds condition on the
difference $c'$ after $N-2$ rounds. It is a function of $x_0$ and the first $N-3$ keys, so by the argument of
Proposition~\ref{prop:markov} the state after $N-2$ rounds is uniform and independent of $c'$ and of the last two keys.
Moreover $c'\ne0$ because every round is a bijection. So the probability is an average of two-round probabilities.
\end{proof}

For every lift $F$ of $p$ the value $\eta$ is $2/q$: the solutions counted by $\delta_F(\alpha,\beta)$, $\alpha\ne0$, lie in
at most two residue classes, and the top layer attains $2q^{k-1}$. With an MDS layer ($B=n+1$) every differential
therefore has probability at most $(2/q)^n$, for instance $2^{-24}$ for $m=7$, $n=4$. This bound does not see the
coefficient choice. Below the top layer it can be separated. For $0\le s\le k-1$ let $D_s$ be the largest entry of the
rows $\delta_F(2^su,\cdot)$, $u$ a unit; so $D_0$ is the largest entry on unit input differences and $D_s$ for $s\ge1$
is the value of Proposition~\ref{prop:layerchar}.

\begin{proposition}[Differentials in a layer]\label{prop:hullunit}
In Proposition~\ref{prop:hull} let $G=F$ be a lift of $p$ and let $a\ne0$ have layer $s$, $0\le s\le k-1$. Then for
every $b$ the probability is at most $(D_s/q^k)^{B-1}$, after two or more rounds.
\end{proposition}
\begin{proof}
Put $c=M^{-1}b$ and $P(u,v)=\delta_F(u,v)/q^k$. As in the proof of Proposition~\ref{prop:hull} the probability is
$\sum_x\prod_iP(a_i,x_i)\prod_jP((Mx)_j,c_j)$. In a nonzero term $v_2(x_i)=v_2(a_i)$ and $v_2((Mx)_j)=v_2(c_j)$ for all
$i$ and $j$, with $v_2(0)=k$ (Proposition~\ref{prop:valuation}). All words of $a$, hence of $x$ and of $Mx$, lie in
$2^sR_k$. So in every nonzero term the supports of the digit vectors $\langle x\rangle_s$ and
$\langle Mx\rangle_s=\ol M\langle x\rangle_s$ are the sets
\[
U=\{i:\ v_2(a_i)=s\},\qquad V=\{j:\ v_2(c_j)=s\},
\]
which depend only on $a$ and $b$, not on $x$. Since $U\ne\emptyset$ and $\ol M$ has branch number $B$, a nonzero term
forces $|U|+|V|\ge B$. Choose $I'\subseteq U$ and $J'\subseteq V$ with $|I'|+|J'|=B-1$. In every nonzero term the
factors indexed by $I'$ and $J'$ have inputs of exact valuation $s$ and are at most $D_s/q^k$, and the rest of the
proof of Proposition~\ref{prop:hull} applies with $\eta$ replaced by $D_s/q^k$. For $N\ge2$ rounds, the difference after
$N-2$ rounds again has layer $s$, since valuations are kept word by word and $\ol M$ is invertible, and the
conditioning argument of Proposition~\ref{prop:hull} applies.
\end{proof}
The proof uses only that $F$ is a bijection with $v_2(F(x+\alpha)-F(x))=v_2(\alpha)$ for all $x$ and $\alpha\ne0$.

The fixed output difference is what makes the sets $U$ and $V$ independent of $x$. Without valuation preservation one
would have to split the sum according to the support of $\langle Mx\rangle_s$, at the price of a factor up to $2^n$.
In the example ($m=7$, $k=3$, $n=4$, $B=5$) the bounds of Proposition~\ref{prop:hullunit} are:
\begin{center}
\begin{tabular}{lccc}
\toprule
 & $s=0$ (unit word) & $s=1$ & $s=2$ (top)\\ \midrule
 $F^\sharp$ of $P_{7,5}$ & $(2/q^2)^4=2^{-52}$ & $(1/q)^4=2^{-28}$ & $(2/q)^4=2^{-24}$\\
 natural lift of $P_{7,5}$ & $(1/q)^4=2^{-28}$ & $(2/q)^4=2^{-24}$ & $(2/q)^4=2^{-24}$\\ \bottomrule
\end{tabular}
\end{center}
They use $D_0=2q$ for $F^\sharp$ (Theorem~\ref{thm:oneterm}) and $D_0=q^2$ for the natural lift (Corollary~\ref{cor:natural}
and Lemma~\ref{lem:ordinary}), $D_1=q^2$ and $2q^2$ (Corollary~\ref{cor:interfamily}), and $D_2=2q^2$ for every lift
(Proposition~\ref{prop:layerchar}). So the coefficient choice changes the differential bound in every layer below the
top. These are bounds for single differentials $(a,b)$; the characteristic bounds $2^{-260}$ and $2^{-140}$ above are a
different quantity, and the bound over all differentials stays $2^{-24}$.

\begin{proposition}[Fixed keys]\label{prop:fixedkey}
In the setting of Proposition~\ref{prop:hull}, fix $a\ne0$ and $b$. For a key pair $\kappa=(\kappa_1,\kappa_2)$ let
$P_\kappa$ be the proportion of $x_0$ for which $(x_0,x_0+a)$ has output difference $b$; it depends only on $\kappa_1$.
The average of $P_\kappa$ over all key pairs is the probability of Proposition~\ref{prop:hull}. Hence, for every
$t>0$, at most a proportion $\eta^{B-1}/t$ of the key pairs satisfies $P_\kappa\ge t$, and at most
$(D_s/q^k)^{B-1}/t$ under the hypotheses of Proposition~\ref{prop:hullunit}.
\end{proposition}
\begin{proof}
Averaging over $\kappa$ counts the triples $(x_0,\kappa_1,\kappa_2)$, which gives the first statement. The second is
Markov's inequality.
\end{proof}

Let a fixed differential have average probability at most $\varepsilon$, for instance $\varepsilon=\eta^{B-1}$ or
$\varepsilon=(D_s/q^k)^{B-1}$, and let $0<\delta<1$. Taking $t=\varepsilon/\delta$, all key pairs outside a set of
proportion at most $\delta$ give this differential probability less than $\varepsilon/\delta$. So a small average
bound gives, for each fixed differential, a bound valid for all key pairs outside a small proportion. It does not give
the bound $\varepsilon$ itself for each key, and it is not a simultaneous assertion for all differentials.

\paragraph{The residue partition.}
The following elementary fact fixes what the bound can and cannot certify.

\begin{proposition}[Residue partition]\label{prop:congpart}
Let $E$ be the map $x\mapsto x+\kappa_0$ followed by $N$ rounds $x\mapsto M\cdot S(x)+\kappa_i$ on $R_k^n$, where $S$
applies any polynomial $G\in R_k[X]$ to each word and $M$ is any additive map of $R_k^n$. For every choice of keys,
$x\equiv y\pmod{2R_k^n}$ implies $E(x)\equiv E(y)\pmod{2R_k^n}$. Hence, for two distinct inputs in the same residue
class, the outputs lie in the same residue class with probability $1$ for every key, while for a uniformly random
permutation of $R_k^n$ this probability is $(q^{(k-1)n}-1)/(q^{kn}-1)<q^{-n}$.
\end{proposition}
\begin{proof}
In $R_k$, $x-y$ divides $G(x)-G(y)$, so $2\mid x-y$ implies $2\mid G(x)-G(y)$ in every word. An additive map sends
$2R_k^n$ into itself, and a common key cancels in the difference. For a random permutation, given the image of $x$
the image of $y$ is uniform among the other $q^{kn}-1$ elements, and $q^{(k-1)n}-1$ of them lie in the residue class
of the image of $x$.
\end{proof}

The event of Proposition~\ref{prop:congpart} is the union of the $q^{(k-1)n}-1$ nonzero output differences in
$2R_k^n$, each of which obeys the per-differential bounds of Corollary~\ref{cor:crypto}, so there is no contradiction.
It is caused by the polynomial structure itself, for every
S-box and every coefficient choice, in line with the observation of~\cite{RS} that non-polynomial maps may be
preferable. This is a structural property, not a key-recovery attack. It says that a complete cipher must contain a component that does
not preserve congruence modulo~$2$: a non-polynomial permutation of $R_k$ or a linear layer that feeds higher
$2$-adic digits into the residue, such as the digit exchange over $\GR(4,m)$ studied in the author's separate manuscript mentioned in the introduction. Not every
change of group suffices: a key added by XOR on the binary digits of the coefficients preserves the partition,
since the lowest digit of the sum depends only on the lowest digits. Our results then apply to the segments built
from ring operations. Inside such a segment each unit-active S-box is crossed with probability at most $2/q^{k-1}$ for
$F^\sharp$, against $4/q^{k-1}$ or $q^{2-k}$ and more for the natural lift. If the segment starts from a uniform
state, adds an independent uniform key in $R_k$ in each of its rounds and is entered by a difference with a unit
component, Proposition~\ref{prop:markov} and Corollary~\ref{cor:crypto}, with $N$ the number of rounds of the
segment, bound characteristics through it; a bound for the whole cipher would in addition need a Markov property of
the other components. The positive content is therefore a design criterion with a sharp target (Theorem~\ref{thm:minimax}), a
certificate that is checkable modulo $8$ (Theorem~\ref{thm:spectrum} and Corollary~\ref{cor:crosslevel}), and, for
$P_{m,r}$, an optimal lift at the price of one multiplication (Theorem~\ref{thm:oneterm}). It is not a security claim
for a complete cipher.

\section{Representatives, applications and open problems}\label{sec:discussion}

\paragraph{Representatives with $\mu=0$ and $Z=\emptyset$.}
The best cell of the minimax table needs a representative with $\mu=0$ and $Z=\emptyset$. Random searches over
$p=L(A(x)^3)$ with $\F_2$-linear $A,L$ find such representatives for $m=5,7,9,11$ (for $m=7$, $238$ among
$2\cdot10^6$ trials; for $m=9$, $81$ among $3\cdot10^6$; for $m=11$, $7$ among $10^6$). The following family settles
their existence for every odd $m\ge5$.

\begin{theorem}\label{thm:repfamily}
Let $m\ge5$ be odd, $T_m=\sum_{i=0}^{m-1}X^{2^i}$, and let $d\in\F_q^*$ satisfy $\Tr(d)=\Tr(d^{-1})=0$. Put
$A_d=T_m+dX^{2^{m-2}}$ and let $p_d$ be the reduced polynomial of the function $x\mapsto A_d(x)^3$. Then $p_d$
satisfies (H) with $\rho=d^{-4}$, $\mu=0$ and $Z=\emptyset$. The number of such $d$ is $(q-3+\mathcal K_m)/4$ with
$\mathcal K_m=\sum_{x\in\F_q^*}(-1)^{\Tr(x+x^{-1})}$; this number is positive for every odd $m\ge5$.
\end{theorem}
\begin{proof}
Write $A=\sum_ia_iX^{2^i}$ for an invertible $\F_2$-linear map, indices modulo $m$. As a function
$A(x)^3=A(x)^2A(x)=\sum_{i,j}a_i^2a_jx^{2^{i+1}+2^j}$, and the reduced polynomial has the Dembowski--Ostrom coefficients
$c_{uv}=a_{u-1}^2a_v+a_{v-1}^2a_u$ at $X^{2^u+2^v}$ ($u\ne v$) and the linear coefficients $a_{w-2}^2a_{w-1}$ at
$X^{2^w}$. Only $X^{1+2^v}$ and $X$ have odd exponent, so $\varepsilon=a_{m-2}^2a_{m-1}$ and
$K=\sum_{v\ge1}c_{0v}X^{2^v}$, which as a function equals $a_{m-1}^2A+a_0A^2=a_0A(A+r)$ with $r=a_{m-1}^2/a_0$ when $a_0a_{m-1}\ne0$ (as below). Hence
$\ker K=A^{-1}\{0,r\}$, so $\rho=A^{-1}(r)$, and $K(\F_q)=a_0r^2\{v^2+v\}=a_0r^2\ker\Tr$. In $p(X+\rho)+p(X)$ only the
terms $c_{1v}X^{2+2^v}$ contribute to $X^2$, which gives $\mu=\sum_{v\ne1}c_{1v}\rho^{2^v}=a_0^2r+a_1r^2$. With
$w=A(a)$ we have $p(a+\rho)+p(a)=(w+r)^3+w^3$ and $p'(a)=\varepsilon+a_0(w^2+rw)$, so
$\Lambda(a)=(r+a_0\rho)(w^2+rw)+r^3+\rho\varepsilon$.

For $A=A_d$, $m\ge5$ gives $a_0=a_1=a_{m-1}=1$ and $a_{m-2}=1+d$, so $r=1$ and $\mu=1+1=0$. If $A_d(x)=0$ then
$dx^{2^{m-2}}=\Tr(x)$; $\Tr(x)=0$ forces $x=0$, and $\Tr(x)=1$ forces $x=d^{-4}$ and $\Tr(d^{-1})=1$. So $A_d$ is
invertible, $p_d$ is a reduced quadratic APN permutation (the cube is one for odd $m$), and
$A_d(d^{-4})=\Tr(d^{-1})+1=1=r$ gives $\rho=d^{-4}$. Next $\varepsilon=(1+d)^2$ and $K(\F_q)=\ker\Tr$, so
$\varepsilon\notin K(\F_q)$ because $\Tr(1+d)=1$; thus (H) holds. Finally $r+\rho=1+d^{-4}\ne0$ since $d\ne1$, and
\[
\frac{r^3+\rho\varepsilon}{(r+\rho)r^2}=\frac{d^4+d^2+1}{d^4+1}=1+\Bigl(\frac{1}{d+1}+\frac{1}{(d+1)^2}\Bigr)^2
\]
has trace $1$, so the constant of the affine map $\Lambda$ is not in the image of its linear part, $(1+\rho)\ker\Tr$,
and $Z=\emptyset$. For the count,
$\frac14\sum_{d\ne0}(1+(-1)^{\Tr d})(1+(-1)^{\Tr d^{-1}})=\frac14(q-3+\mathcal K_m)$, and the Weil bound
$|\mathcal K_m|\le2\sqrt q$ for Kloosterman sums makes this positive for $q\ge32$.
\end{proof}

Within the family the conditions are also necessary: $A_d$ is singular when $\Tr(d^{-1})=1$, and
$\varepsilon\in K(\F_q)$ when $\Tr(d)=1$. They coincide with the admissibility conditions of the trace family below.
If all coefficients of $p$ lie in $\F_2$, then $\rho=1$, $\varepsilon=1$ and $p(0)+p(1)=1$ because $p$ permutes
$\F_2$, so $\Lambda(0)=0$ and $Z\ne\emptyset$; representatives with $\mu=0$ and $Z=\emptyset$ therefore need coefficients
outside $\F_2$. The number of admissible $d$ is $10,28,126,528,2002$ for $m=5,\dots,13$. The field-level part of the proof (invertibility of $A_d$,
$A_d(d^{-4})=1$, $\ker K=\{0,d^{-4}\}$, $\varepsilon\notin K(\F_q)$ and the absence of zeros of $\Lambda$) is formalized
(Appendix~\ref{app:verification}); the coefficient formulas for $p_d$ and the count are on paper. We checked
Theorem~\ref{thm:repfamily} for every admissible $d$ with $m\le13$ by evaluating $p_d$ at all points (including a full
difference table for $m\le9$), for sampled $d$ up to $m=31$ symbolically, and for several $d$ with the independent
invariant program of Appendix~\ref{app:verification}. Combined with Theorem~\ref{thm:minimax}, every odd $m\ge5$ thus
has lifts with $\Dexc=q$ at every level $k\ge3$. The representatives $p_d$ are not sparse: $p_d$ has $3m-3$
monomials. A second family has a bounded number of monomials.

\begin{theorem}[A sparse family]\label{thm:sparse}
Let $m\ge5$ be odd, $c\in\F_q\setminus\F_2$, $B_c=X^{q/4}+X^{q/2}+X+X^2+cX^4$, and let $p_c$ be the reduced polynomial of
$x\mapsto B_c(x)^3$; it has $16$ monomials for $m\ge7$, $12$ for $m=5$, and degree $3q/4$. If $B_c$ is injective, then
$p_c$ satisfies (H) with $\varepsilon=1$, $\rho=B_c^{-1}(1)$, $\mu=0$ and $Z=\emptyset$. For every odd $m\ge5$ at least two
such $c$ exist; there are $15,42,126,506,2496$ of them for $m=5,7,9,11,13$.
\end{theorem}
\begin{proof}
In the proof of Theorem~\ref{thm:repfamily} take $a_{m-2}=a_{m-1}=a_0=a_1=1$, $a_2=c$ and $a_i=0$ otherwise. Then
$r=1$, $\varepsilon=1\notin K(\F_q)=\ker\Tr$, and $\mu=a_0^2r+a_1r^2=0$, so (H) holds. Since $B_c(1)=c\ne1$ we have
$\rho\ne1$, and $\Lambda(a)=(1+\rho)(w^2+w)+1+\rho=(1+\rho)(w^2+w+1)$ with $w=B_c(a)$; for odd $m$ the polynomial $w^2+w+1$
has no root in $\F_q$, so $Z=\emptyset$.

For existence, write $x\ne0$ as $x=y^4$ and put $z=y^{-1}$. Then $B_c(x)=y+y^2+y^4+y^8+cy^{16}$ vanishes exactly when
$c=h(z)$ with $h=Z^{15}+Z^{14}+Z^{12}+Z^8$. So every $x\ne0$ lies in the kernel of exactly one $B_c$, and counting
$\sum_c(|\ker B_c|-1)=q-1$ gives $\#\{c:B_c\text{ injective}\}=1+\sum_{\ker B_c\ne0}(|\ker B_c|-2)$. This is at least $3$
as soon as some kernel has dimension at least $2$, that is, as soon as $h(z)=h(w)$ for some $z\ne w$ in $\F_q^*$. Since
$B_0$ is singular and $\F_2$ contains only one further value, at least two $c$ then remain. The pairs with $h(z)=h(w)$,
$z\ne w$, are the points of the plane curve $Q(z,w)=(h(z)+h(w))/(z+w)=0$ of degree $14$; neither $z=w$ (where
$Q=z^{14}$) nor the axes are components. Over $\F_{2^{17}}$ this curve has $125769$ affine points. If no absolutely
irreducible component of $Q$ were defined over $\F_2$, the Frobenius orbits of the components would have lengths
between $2$ and $14$, none dividing the prime $17$, so every $\F_{2^{17}}$-point would lie on two distinct components, namely on a component $Y$ and on its conjugate $\sigma^{17}(Y)\ne Y$,
and by Bézout there would be at most $14^2/2=98$ such points. Hence $Q$ has an absolutely irreducible component $Y$
of degree $e\le14$ defined over $\F_2$. By the Weil bound for singular curves of Aubry and Perret~\cite{AP},
$Y$ has at least $q+1-(e-1)(e-2)\sqrt q$ points over $\F_q$ on its projective closure; removing at most $e$ points at
infinity and at most $3e$ points on the lines $z=0$, $w=0$, $z=w$ leaves a positive number for $q\ge2^{15}$. The
values for $m\le13$ were computed exhaustively.
\end{proof}

The counts for $m\le13$ and the point count over $\F_{2^{17}}$ were each obtained by two independent programs. We
checked Theorem~\ref{thm:sparse} symbolically for every admissible $c$ with $m\le13$ and for samples up to $m=31$.

\paragraph{Optimal lifts at no extra cost.}
For a representative with $\mu=0$, optimality at every level only asks that $A_F$ have no zero
(Theorems~\ref{thm:minimizers} and~\ref{thm:CL}), and one coefficient digit suffices.

\begin{proposition}\label{prop:lowcost}
Let $p$ satisfy (H) with $\mu=0$, and let $F$ be a lift of $p$ whose exponents have binary weight at most $2$, for
instance any lift with the monomials of $p$. Then $A_F=A_F(0)+L_p$ with an $\F_2$-linear map $L_p$ that does not depend on $F$ and
satisfies $L_p(\rho)=0$. Let $j\ge1$ be such that $X^{1+2^j}$ has a nonzero coefficient in $p$ (such $j$ exists since $K\ne0$), let
$\tilde b\in R_k$ be its coefficient in $F$, with residue $b$, and let $F_u$ be $F$ with $\tilde b$ replaced by $\tilde b+2u$. Then $A_{F_u}=A_F+\ol u\rho^{2^j}$, and $F_u$ is optimal at every level $k\ge3$ for at least half of
the residues $\ol u$. It has the same monomials and degree as $F$, and if $b\notin\F_2$ it costs no extra operation in the termwise
evaluation model of Section~\ref{sec:family}.
\end{proposition}
\begin{proof}
Compute $A_F(a)$, the residue of $(F'(x+\tau)-F'(x))/2$ with $\ol x=a$, $\ol\tau=\rho$, monomial by monomial; let $bX^e$ be a monomial of $F$. Exponents
$e=1$ and $e\equiv0\pmod 4$ contribute no $x$-dependent term, $e=2+2^v$ ($v\ge2$) the term
$\ol b(a\rho^{2^v}+\rho a^{2^v})$, and, since $(x+\tau)^{2^v}\equiv
x^{2^v}+\tau^{2^v}+2(x\tau)^{2^{v-1}}\pmod4$, $e=1+2^v$ ($v\ge1$) the term $\ol b(a\rho)^{2^{v-1}}$. The
$x$-independent parts of all monomials add up to an element of $2R_k$, because the total and the $x$-dependent parts
do, so they contribute a constant. Hence
$A_F(a)=A_F(0)+L_p(a)$ with $L_p(a)=K(\sqrt{a\rho})+\rho K_1(a)+\mu a$, where $c_{uv}=c_{vu}$ is the coefficient of $X^{2^u+2^v}$ in $p$ and $K_1=\sum_{v\ne1}c_{1v}X^{2^v}$; the
coefficients of $p$ alone determine $L_p$, and $L_p(\rho)=K(\rho)+\rho K_1(\rho)+\mu\rho=0$ because $K_1(\rho)=\mu$. So
$A_F(\F_q)$ has at most $q/2$ elements. Proposition~\ref{prop:Acontrol}(i) with $H=uX^{1+2^j}$ adds the constant
$\ol u\rho^{2^j}$, and $A_{F_u}$ has no zero if and only if $\ol u\rho^{2^j}\notin A_F(\F_q)$.
\end{proof}

For $P_{m,r}$ with $r\ge5$ we have $\rho=1$, and the coefficient $1$ of $X^{2^{r-2}+1}$ replaced by $3$ gives
$A=A_{F^\sharp}$: an optimal lift with exactly the monomials of the natural lift, at the cost of two additions and no
multiplication in the termwise model, one multiplication less than $F^\sharp$ ($F^\sharp$ is the case $j=1$, where $X^3$ is
a new monomial). Theorems~\ref{thm:O3}, \ref{thm:M4} and~\ref{thm:inter}(ii),(iii) depend on the lift only through $A_F$ (given
$c_3$, $\pi$ and $\mu$), so the proof of Corollary~\ref{cor:interfamily} applies verbatim and this lift is also optimal
in every layer below the top.
For $r\ge5$ it therefore improves on $F^\sharp$ in cost and gives the same exceptional and layer spectra. For $p_c$ the
coefficient $1+c$ of $X^5$ ($j=2$) gives optimal lifts with the monomials of $p_c$ ($16$ for $m\ge7$, $12$ for $m=5$) and
no extra operation; the digit is found by linear algebra. With $f=t^5+t^2+1$, $t^7+t+1$ and $t^9+t^4+1$, the natural lift itself is optimal for $13$ of the $15$
values of $c$ at $m=5$, $28$ of $42$ at $m=7$ and $83$ of $126$ at $m=9$; these counts depend on the convention for
lifting coefficients. At $m=5$, $k=3$ we computed all exceptional rows of the $13$ optimal natural lifts and of two
digit-changed lifts $F_u$, and those of a further explicit optimal lift (given in the notes in Online Resource~1) for all
$15$ values of $c$, and at $m=7$, $k=3$ all $16384$ rows of that lift for two values of $c$: every exceptional entry is
$0$ or $q$, as predicted. Natural lifts and digits $\ol u$ with zeros of $A$ give entries $q^2$, also as predicted.

\paragraph{The trace family of Bartoli and Stănică.}
Let $m\ge5$ be odd, let $T_m=\sum_{i=0}^{m-1}X^{2^i}$ be the absolute trace polynomial and let $w\in\F_q^*$ satisfy
$\Tr(w)=\Tr(w^{-1})=0$. The polynomials $F_w=X^3+wX^2T_m+w^2XT_m+w^3T_m$ of~\cite[Thm.~2.4]{BS} satisfy (H): they are
reduced and quadratic, and
$F_w'=w^3+K$ with $K=(1+w)A_w^2+w^2A_w$, $A_w=X+wT_m(X)$, $\ker K=\{0,\rho\}$ and $\rho=A_w(w^2/(1+w))$. For $m=5$ and
$w=t$ (degree $18$) we find $\rho=31$, $\mu=13$ and $Z=\emptyset$. So this member lies in the cell where the optimum
is $q$ at $k=3$ and $2q$ at $k\ge4$, and no family is optimal at all levels (Corollary~\ref{cor:CL2}). Indeed
its natural lift has $A$ zero-free and exceptional rows ``$64$ appears $512$ times'' at $k=3$. It is optimal at
every $k\ge4$ by Theorem~\ref{thm:minimizers}. The lift with $A\equiv0$ has rows ``$32$ appears $1024$ times''
at $k=3$, but its rows reach $q^2=1024$ at $k=4$. For $m=5,7,9$ every admissible $w$ gives $\mu\ne0$. Half of
them have $Z=\emptyset$, and the other half have $|Z|=2$. The constant-derivative family
of~\cite[Thm.~2.3]{BS} lies outside (H), since it has no exceptional direction.

\paragraph{Open problems.}
(1) Closed formulas for the parameters $d$ and $c$ of Theorems~\ref{thm:repfamily} and~\ref{thm:sparse}, and for a good
digit $\ol u$ in Proposition~\ref{prop:lowcost}, that is, a closed-form optimal lift without new monomials. (2) Closed
formulas for the residual configurations of the intermediate layers,
and expressions of the level data of $T$ through the coefficients of $F$. (3) The layer values $D_s$ of
Proposition~\ref{prop:layerchar} for $r=3$ and for general lifts. (4) Per-key bounds for single differentials beyond the
averaging argument of Proposition~\ref{prop:fixedkey}.

\section*{Statements and Declarations}

\paragraph{Funding.} No funding was received for conducting this study.

\paragraph{Competing interests.} The author has no competing interests to declare that are relevant to the content
of this article.

\paragraph{Data and code availability.} The programs, inputs, logs and result files of the finite computations of
Appendix~\ref{app:verification}, the replay scripts \texttt{replay\_core.sh} and \texttt{replay\_full.sh} with their
negative controls, and the Lean~4 sources with their build, audit and kernel re-check records are provided as
Online Resource~1 (archive \texttt{ESM\_1.zip}). The core replay takes about six minutes on two cores and needs gcc and
Python~3 with SymPy and NumPy; the full replay takes about one and a half hours. The same material, with the blueprint of the formalization, is archived at
\url{https://doi.org/10.5281/zenodo.23683493} (source repository \url{https://github.com/hirokifukui/residue-apn},
tag \texttt{v2.0.0}).

\paragraph{Author contribution.} The author is the sole author of this work and carried out all parts of it.

\paragraph{Use of generative AI.} AI tools were used in this work as described in Appendix~\ref{app:verification}.

\appendix
\section{Verification}\label{app:verification}

\paragraph{Use of AI tools.}
Claude (Anthropic), including Claude Code, and ChatGPT (OpenAI) were used to explore and draft proofs, to write the
Lean code and the programs for the finite computations, to draft and revise the text, and to review successive drafts.
The model versions changed during the project; representative prompts are available on request. The author made all
mathematical and editorial decisions, checked every definition, statement, proof and citation, and takes sole
responsibility for the content. The correctness of the formalized results rests on the Lean kernel, and the reported
spectra were cross-checked by independently written programs; neither check depends on accepting an assertion made by
an AI tool.

\paragraph{Computations.}
Exceptional rows were counted by a C program over $\GR(2^k,m)$ and compared with the predictions of
Theorems~\ref{thm:S12} and~\ref{thm:S3} by a fail-closed validator. The validator rejects empty or unsupported
scopes and pins the expected record counts. Rows outside the scope of the spectrum theorems are checked
only against the proven lower bound $q^2$ and are not counted as verified spectra. For $m=5$ all $1024$
exceptional rows at $k=3$ and sampled rows at $k=4$ were checked for $16$ lifts at $k=3$ and $k=4$, giving $32$ records: $27$ spectrum matches and $5$
records outside the spectrum theorems that were checked against the lower bound $q^2$. For $m=7$ all $16384$ exceptional rows were checked for the representative above,
for $F^\sharp$ and the natural lift of $P_{7,3}$, and for $F^\sharp$ and the natural lift of $P_{7,5}$ (whose rows
reach exactly $q^2$, with $n_{2q}=8064$ and $n_{q^2}=2$ as in Corollary~\ref{cor:natural}). For $m=9$ two sampled
exceptional rows each of the natural lifts of $P_{9,5}$ and $P_{9,7}$ over $\GR(8,9)$ were computed in full
($2^{27}$ inputs per row); they are ``$2^{10}$ appears $130560$ times, $2^{18}$ twice'' and ``$2^{10}$ appears $130560$
times, $2^{19}$ once'', as predicted. An independent Magma implementation agrees on $64$ distinct directions for four
lifts. The two-word normalization of Proposition~\ref{prop:markov} was checked exhaustively on a two-word model over
$\Z/8$ ($258048$ characteristics), and the bound of Proposition~\ref{prop:hullunit} exhaustively on two words over
$\GR(4,2)$ for $32$ random valuation-preserving permutations (all input and output differences, exact integers). The
replay scripts are fail-fast; they reject empty, unknown or repeated stage
lists, report a partial run as partial, and are tested by negative controls.

\paragraph{Lean.}
The central results are formalized in Lean~4 (Mathlib) over an abstract axiom system for Galois rings, and
$\GR(2^k,m)=(\Z/2^k)[X]/(f)$ is shown to satisfy it. The formalization covers the invariants and their control,
the lower bounds, the minimax theorem and all minimizers ($k=3$ and $k\ge4$), the spectra, all layer results
(Theorems~\ref{thm:O3}, \ref{thm:M4}, \ref{thm:inter} with the multiplicities of (ii), Corollaries~\ref{cor:M5}
and~\ref{cor:interfamily}), the
residue APN and permutation properties of $P_{m,r}$, Theorem~\ref{thm:oneterm} on all unit directions together with
its degree statement, the natural-lift benchmarks of Theorem~\ref{thm:oneterm} ($\Dexc=4q$ for $r=3$ at every level,
an entry $\ge q^2$ in every exceptional row for $r\ge5$ at every level, and the full spectrum of
Corollary~\ref{cor:natural}), the literal comparison of difference tables across rings in Theorem~\ref{thm:S3},
Proposition~\ref{prop:markov} for an arbitrary finite additive commutative group, S-box and additive automorphism,
part (a) of Corollary~\ref{cor:crypto} (for any lift whose unit rows are bounded by $D$, any $M$ whose reduction has
branch number $B$, and any characteristic with a unit input component, the probability is at most
$(D/|R_k|)^{B\lfloor N/2\rfloor}$; for $F^\sharp$ with $D=2q$ and no hypotheses beyond those of the corollary), and the congruence statement of
Proposition~\ref{prop:congpart} for any commutative ring, polynomial S-box, additive layer and key schedule (the
comparison with a random permutation is on paper), the field-level part of Theorem~\ref{thm:repfamily},
Proposition~\ref{prop:layerchar} (for any lift whose derivative has no zero and any bound $D_s$, the top-layer value for
every such lift, and $F^\sharp$ with $r\ge5$ in every layer), Proposition~\ref{prop:hull} for two rounds over any finite
abelian group of words, the transfer of the branch number of $\ol M$ to $R_k^n$, Proposition~\ref{prop:hullunit} for two
rounds (for any polynomial S-box that permutes $R_k$, reduces to a permutation and has $p'$ zero-free, in every layer
$s$ and with any bound $D_s$), and Proposition~\ref{prop:fixedkey}; the instances $\eta=2/q$, $D_0=2q$ and
$D_s=q^{s+1}$ ($1\le s\le k-2$) are formalized for $F^\sharp$ with $r\ge5$, and the top-layer instance
$D_{k-1}=2q^{k-1}$ for every lift whose $p'$ is two-to-one.

Optimality across levels is formalized in two forms. Theorem~\ref{thm:CL} and the existence statement for one fixed
pair of levels are formalized abstractly and on $\GR(2^k,m)\to\GR(8,m)$. Corollary~\ref{cor:CL2}, with its quantifier
over all levels, is formalized for an abstract tower of Galois rings over a common coefficient ring, and this tower
is constructed for the concrete rings: for a monic $g\in\Z[X]$ with irreducible reduction modulo $2$, the rings
$\GR(2^{j+3},m)=(\Z/2^{j+3})[X]/(g)$, $j\ge0$, with coefficients taken from $\Z[X]/(g)$, satisfy its hypotheses, so
Corollary~\ref{cor:CL2} holds formally on $\GR(2^k,m)$ for all levels at once. All audited declarations depend only on the standard axioms,
and the compiled modules were re-checked by the kernel replay tool \texttt{leanchecker}. Not formalized: the operation
count of Theorem~\ref{thm:oneterm}, the passage from two to $N\ge2$ rounds in Propositions~\ref{prop:hull}
and~\ref{prop:hullunit}, the value $\eta=2/q$ for a general lift, the statement of
Theorem~\ref{thm:oneterm} on which monomials occur (its degree statement is formalized), the existence part and the
coefficient formulas of Theorem~\ref{thm:sparse} (its finite certificate, the irreducibility of $t^{17}+t^3+1$, the
$125769$ points and the counts for $m\le13$, is recomputed by a separate program in the replay),
Proposition~\ref{prop:lowcost} for a general $p$ (for $P_{m,r}$ the equality of $A_F$ with that of $F^\sharp$ and the
resulting unit-direction bound are formalized; its layer optimality is on paper), Proposition~\ref{prop:GR} and the classification of the intermediate layers built on it, and the
finite computations. A theorem-level scope table is
part of Online Resource~1.

\end{document}